\documentclass[11pt,a4paper]{article}
\usepackage{amsfonts,latexsym,amsmath, amssymb, amscd, amsthm}
\usepackage{mathrsfs}
\usepackage{bm}
\usepackage{appendix}
\usepackage{color}
\usepackage{geometry}
\usepackage{cite}
\usepackage{cases}
\usepackage{ulem}
\usepackage[utf8]{inputenc}
\usepackage{graphicx,enumerate}
\usepackage[T1]{fontenc}
\usepackage{authblk}
\numberwithin{equation}{section}

\def\varep{\varepsilon}

\newtheorem{theorem}{Theorem}[section]
\newtheorem{lemma}[theorem]{Lemma}
\newtheorem{proposition}[theorem]{Proposition}

\newtheorem{definition}[theorem]{Definition}
\newtheorem{remark}{Remark}

\allowdisplaybreaks

\title{Spreading speeds and pulsating fronts for a field-road model in a spatially periodic habitat}
\author{
	Mingmin Zhang\thanks{mingmin.zhang.math@gmail.com (M. Zhang).}
	\\
	{\small
		School of Mathematical Sciences, University of Science and Technology of China, Hefei, Anhui 230026, China\\
		Aix Marseille Universit\'{e}, CNRS, Centrale Marseille, I2M,  
		 Marseille, France
	}
}
\date{}

\begin{document}
	
	\maketitle

	\begin{abstract}

	A reaction-diffusion model which is called the field-road model was introduced by Berestycki, Roquejoffre and Rossi \cite{BRR2013-1} to describe biological invasion with fast diffusion on a line. In this paper, we investigate  this model in a heterogeneous landscape and establish the existence of the asymptotic  spreading speed $c^*$ as well as its coincidence with the minimal wave speed of pulsating fronts along the road.  We start with  a truncated problem with an imposed  Dirichlet boundary condition. We prove the existence of spreading speed $c^*_R$ which coincides with the minimal speed of pulsating fronts for the truncated problem in the  direction of the road. The arguments combine the dynamical system method with PDE's approach. 
	Finally, we turn back to the original problem in the half-plane via generalized principal eigenvalue approach as well as an asymptotic method. 
	\vspace{2mm}
	
   	\vspace{1mm}
   
		\end{abstract}
	\section{Introduction}
	
The goal of this paper is  to investigate propagation properties for a  field-road model in a spatially periodic environment. Taking into account this  heterogeneity in space, we shall establish the existence of the asymptotic spreading speed and its coincidence with the minimal wave speed of pulsating traveling fronts in the direction of the road. In this paper, the line $\{(x,0):~ x\in\mathbb{R}\}$ will be referred to as the road in the plane $\mathbb{R}^2$. The heterogeneity is assumed to appear in $x$-direction. Then by symmetry, we can consider the upper half-plane $\Omega:=\{(x,y)\in\mathbb{R}^2:~ y>0\}$ as the field.  Denote by $u(t,x)$  the density of population on the  road and by $v(t,x,y)$ the density of  population in the field. The population in the field is assumed to be governed by a Fisher-KPP equation with diffusivity $d$ and heterogeneous nonlinearity $f(x,v)$, whereas  the population on the road is subject to a diffusion equation with diffusivity $D>0$ which is a priori different from $d$. Moreover, there are exchanges of populations between the road and the field in which the parameter $\mu>0$ stands for the rate of individuals on the road going into the field, while the parameter $\nu>0$  represents the rate of individuals passing from the field to the road. Therefore, we are led to the following system:
	\begin{equation}
	\label{pb in half-plane}
	\begin{cases}
	\partial_t u -D\partial_{xx}u=\nu v(t,x,0)-\mu u, & t>0,\ x\in \mathbb{R},\\
	\partial_t v -d\Delta v =f(x,v), &t>0,\ (x,y)\in\Omega,\\
	-d\partial_y v(t,x,0)=\mu u-\nu v(t,x,0), &t>0,\  x\in \mathbb{R}.
	\end{cases}
	\end{equation} 
 We assume that the reaction term $f(x,v)$ depends on the $x$ variable in a periodic fashion. As a simple example, $f$ may be of the type $f(x,v)=a(x)v(1-v)$
in which the periodic coefficient $a(x)$  can be interpreted as an effective birth rate of the population. In models of biological invasions, the heterogeneity may be a consequence of the presence of highly differentiated zones such as forests, rivers, grasslands, roads, villages, etc., where the species in consideration may tend to reproduce or die with different rates from one place to another. Therefore, it is a fundamental problem to understand how heterogeneity influences the characteristics of front propagation such as front speeds and front profiles.

Let us recall the origin of this model and relevant results.	The field-road model was first introduced by Berestycki, Roquejoffre and Rossi \cite{BRR2013-1} in 2013 where all parameters are homogeneous. The authors proved that a strong diffusion on the road enhances global invasion in the field. More precisely, denote by $w^*$ the asymptotic  spreading speed in the direction of the road for the homogeneous field-road model and by $c_{KPP}:=2\sqrt{df'(0)}$ the spreading speed for the scalar KPP equation $u_t-d u_{xx}=f(u)$, they proved that:	 if $D\le 2d$, then $w^*=c_{KPP}$; if $D>2d$, then $w^*>c_{KPP}$. Moreover, they showed that the propagation velocity on the road increases indefinitely as $D$ grows to infinity. 	As a sequel,	the same authors introduced in \cite{BRR2013-2}  transport and mortality on the road  to understand the resulting new effects.   Let us point out that the original model was considered in a homogeneous frame, which means that every place in the field is equivalently suitable for the survival of  species, whereas this homogeneity assumption is hardly satisfied in natural environments. Therefore, it is of the essence to take into account the heterogeneity of the medium. Later on, it was proved  in \cite{BCRR2013-2014,BRR2016-1} that the road enhances the asymptotic speed of propagation in a cone of directions.	The paper \cite{BRR2016-2} established the existence of standard traveling fronts for this homogeneous system for $c\ge w^*$. Giletti, Monsaingeon and Zhou \cite{GMZ2015}  considered this model with spatially periodic exchange coefficients:
	\begin{equation*}
	\begin{cases}
	\partial_t u -D\partial_{xx}u=\nu(x) v(t,x,0)-\mu(x) u, &t>0,\  x\in \mathbb{R}, \\
	\partial_t v -d\Delta v =f(v), &t>0,\  (x,y)\in\Omega,\\
	-d\partial_y v(t,x,0)=\mu(x) u(t,x)-\nu(x) v(t,x,0), &t>0,\  x\in \mathbb{R},
	\end{cases}
	\end{equation*}
	where $\mu(x), \nu(x)$ are $L$-periodic in $x$ in $C^{1,r}(\mathbb{R})$, and $\mu(x),\nu(x) \ge\not\equiv 0$. They recovered the same diffusion threshold $D=2d$ in \cite{BRR2013-1}.
	In 2016, Tellini \cite{Tellini2016} studied the homogeneous field-road model in a strip with an imposed Dirichlet boundary condition on the other side of the strip. It is noticed that traveling fronts were just studied in \cite{BRR2016-2} for the original homogeneous model and in \cite{Dietrich2015} for a truncated problem with ignition-type nonlinearity. 
	
	Related results were also obtained in various frameworks.
	The case of a fractional diffusion
	on the road was treated in \cite{BCRR2013-2014,BCRR2015}.   Nonlocal exchanges were studied in \cite{Pauthier2015,Pauthier2016}. Models with an ignition-type nonlinearity were considered in \cite{Dietrich2015,Dietrich2017}.  The field-road model set in an infinite cylinder with fast diffusion on the surface was investigated in \cite{RTV2017}. The case where the field is a cone was studied in \cite{Ducasse2018}.
 The authors in \cite{BDR2019}  discussed the effect of the road on a population in an ecological niche facing climate change based on
the notion of generalized principal eigenvalues for heterogeneous road-field systems developed in \cite{BDR1}. Propagation phenomena for heterogeneous KPP bulk-surface systems in a cylindrical domain  was investigated recently in \cite{BGT2020}. The existence of weak solutions to an elliptic problem in bounded and unbounded strips motivated by the field-road model  was discussed in \cite{CZ2020}. 
An interesting but different field-road model where the road is with very thin width was introduced in  \cite{LW2017} using the so-called effective boundary conditions to study  speed enhancement and the asymptotic spreading speed. 
	 
	  By contrast with standard periodic reaction-diffusion equations, the mathematical study of \eqref{pb in half-plane} contains the following difficulties: firstly, the periodic assumption only set on $x$ variable but not on $y$ leads to the noncompactness of the domain,  therefore the existence of pulsating  fronts cannot be obtained by PDE's methods easily. Secondly, due to the heterogeneous hypothesis on  $f$, the situation is much more involved so that we are not able to derive precise threshold result of speed enhancement with respect to different diffusivities on the road and in the field. Thirdly, in terms of the generalized eigenvalue problem in the half-plane,  one of main technical difficulties is to get some estimates for the generalized principal eigenfunction pair. To the best of our knowledge, there has been no known result about the existence of generalized traveling fronts for the field-road model in heterogeneous media up to now.
	  
	  The aim of this work is to prove the existence of the asymptotic spreading speed $c^*$ as well as its coincidence with the minimal speed of  pulsating traveling fronts along the road for \eqref{pb in half-plane} in a spatially periodic habitat. Our strategy is to study a truncated problem with an imposed zero Dirichlet upper boundary condition  as a first step.
	  Specifically, by application of principal eigenvalue theory and of dynamical system method, we show the existence of the  asymptotic spreading speed $c^*_R$ as well as its  coincidence with the minimal speed of  pulsating traveling fronts along the road.  We further give a variational formula for $c^*_R$ by using the principal eigenvalue of certain linear elliptic problem.   Based on the study of the truncated problem, we eventually go back to the analysis of the  original problem in the half-plane by combining generalized principal eigenvalue approach with an asymptotic method. Let us mention that the results in this paper can also be adapted to the case  of periodic exchange coefficients treated in \cite{GMZ2015}.

	For general reaction-diffusion problems, there have been lots of remarkable works on spreading properties and pulsating traveling fronts. We refer to \cite{Weinberger2002, BH2002, BHN2005, LZ2007, LZ2007, BHN2010, BHR2005-2, HR2011} and references therein.

	\section{Hypotheses and main results}
	Throughout this paper, we assume that $f$: $ \mathbb{R}\times\mathbb{R}_+\to \mathbb{R}$ is of class $C^{1,\delta}$ in $(x,v)$ (with $0<\delta<1$) and $C^2$ in $v$, $L$-periodic in $x$, and satisfies the KPP assumption:
	$$ f(\cdot,0)\equiv 0\equiv f(\cdot,1),~  0<f(\cdot,v)\le f_v(\cdot,0)v~ \text{for}~v\in(0,1),~ f(\cdot,v)<0~ \text{for}~v\in(1,+\infty).$$
	Define  $M:=\max_{[0,L]}f_v(x,0)$ and $m:=\min_{[0,L]}f_v(x,0)$. Then $M\ge m>0$. We further assume that   
	$$\forall x\in\mathbb{R},~ v\mapsto   \frac{f(x,v)}{v} \  \text{is  decreasing in} \  v>0.$$

	In what follows, as far as the Cauchy problem is concerned, we always assume that  the initial condition $(u_0,v_0)$
	is nonnegative, bounded and  continuous.
	

	
We now present our results in this paper. As a first step, we focus on the following truncated problem with an imposed Dirichlet upper boundary condition:
\begin{equation}
\label{pb in strip}
\begin{cases}
\partial_t u -D\partial_{xx}u=\nu v(t,x,0)-\mu u, &t>0,\  x\in \mathbb{R}, \\
\partial_t v -d\Delta v =f(x,v), &t>0,\ (x,y)\in\Omega_R,\\
-d\partial_y v(t,x,0)=\mu u-\nu v(t,x,0), &t>0,\  x\in \mathbb{R},\\
v(t,x,R)=0, &t>0,\  x\in \mathbb{R},
\end{cases}
\end{equation}
in which $\Omega_R:=\{(x,y)\in\mathbb{R}:~ 0<y<R\}$ denotes a truncated domain with width $R$ sufficiently large. In fact, the width $R$ of the strip plays a crucial role in long time behavior of the corresponding Cauchy problem \eqref{pb in strip} due to the zero Dirichlet upper boundary condition. A natural explanation, from the biological point of view, is that if the width of the strip is not sufficiently large, the species may finally extinct because of the effect of unfavorable Dirichlet condition on the upper boundary. Therefore, we shall give  a sufficient condition on $R$ such that the species can persist successfully. Here is our statement.
	\begin{theorem}
		\label{thm2.1}
		If
		\begin{equation}
		\label{R-condition}
		m>\frac{d\pi^2}{4R^2},
		\end{equation}
		then \eqref{pb in strip} admits a unique nontrivial nonnegative stationary solution $(U_R, V_R)$,  which is $L$-periodic in $x$. Moreover, let $(u,v)$ be the solution of \eqref{pb in strip} with a nonnegative,  bounded and continuous initial datum $(u_0, v_0)\not\equiv(0,0)$, then 
		\begin{equation}
		\label{large time-truncated}
		\lim\limits_{t\to +\infty}(u(t,x),v(t,x,y))= (U_R(x), V_R(x,y))~~\text{locally uniformly in}~ (x,y)\in\overline\Omega_R.
		\end{equation}
	\end{theorem}
	\begin{remark}
	\textnormal{In particular, when the environment is  homogeneous, i.e. $f(x,v)\equiv f(v)$, $R$ should satisfy $4R^2 f'(0)>d\pi^2$, which coincides with the condition in \cite{Tellini2016}. Let $R_*>0$ be such that $m=\frac{d\pi^2}{4R^2_*}$.
	For any	$R>R_0:=2R_*$, \eqref{R-condition} is satisfied and there also holds $m=\frac{d\pi^2}{R^2_0}>\frac{d\pi^2}{R^2}$.
	\textit{Throughout the paper, as far as the truncated problem is concerned, it is not restrictive to assume that $R>R_0$ (since our concern is to take $R\to+\infty$ to consider \eqref{pb in half-plane}), which will be convenient to prove the positivity of the asymptotic spreading speed $c^*_R$ for problem \eqref{pb in strip}.}}
\end{remark}

Let $(U_R,V_R)$ be the unique nontrivial nonnegative stationary solution of \eqref{pb in strip} in the sequel. We are now in a position to investigate  spreading properties of solutions to \eqref{pb in strip} in $\overline\Omega_R$, which  is based on dynamical system method and principal eigenvalue theory.

We first consider  the following eigenvalue problem in the strip $\overline\Omega_R$:
\begin{align}
\label{eigen-pb-strip}
\begin{cases}
-D\phi''+2D\alpha\phi'+(-D\alpha^2+\mu)\phi-\nu\psi(x,0)=\sigma\phi,  &x\in\mathbb{R},\\
-d\Delta\psi+2d\alpha\partial_x\psi-(d\alpha^2+f_v(x,0))\psi=\sigma\psi,  &(x,y)\in
\Omega_R,\\
-d\partial_y\psi(x,0)+\nu\psi(x,0)-\mu\phi=0, \   &x\in\mathbb{R},\\
\psi(x,R)=0, \   &x\in\mathbb{R},\\
\phi, \psi \ \text{are} \   L\text{-periodic with respect to} \ x.
\end{cases}
\end{align}
The compactness of the domain allows us to  apply the classical Krein-Rutman theory which provides the existence of  the  principal eigenvalue $\lambda_R(\alpha)\in\mathbb{R}$ and the associated unique (up to multiplication by some constant) positive principal eigenfunction pair $(P_{\alpha,R}(x),Q_{\alpha,R}(x,y)) \in C^3(\mathbb{R})\times C^3(\overline\Omega_R)$ for each $\alpha\in\mathbb{R}$. 
 
    \begin{theorem}
  	\label{thm-asp-strip}
  	Let $(U_R,V_R)$ be the unique nontrivial nonnegative stationary solution of \eqref{pb in strip} obtained in Theorem \ref{thm2.1} and let $(u,v)$  be the solution of \eqref{pb in strip} with a nontrivial continuous initial datum $(u_0,v_0)$ with $(0,0)\le (u_0,v_0)\le  (U_R,V_R)$ in $\overline\Omega_R$. Then there exists $c^*_R>0$ given by
  	\begin{equation*}
  		c^*_{R}=\inf\limits_{\alpha>0}\frac{-\lambda_R(\alpha)}{\alpha},
  	\end{equation*}
  	 called the 
  	  asymptotic spreading speed,
  	  such that the following statements are valid:
  	 \begin{enumerate}[(i)]
  	 	\item If $(u_0,v_0)$ is compactly supported, then for any $c> c^*_{R}$, there holds
  	 	\begin{equation*}
  	 		\lim\limits_{ t\to+\infty}\sup\limits_{|x|\ge ct,\ y\in[0,R]} |(u(t,x),v(t,x,y))|=0.
  	 	\end{equation*}
  	 \item 
  	 For any $0<c<c^*_R$, there holds
  	 	\begin{equation*}
  	 \lim\limits_{ t\to+\infty}\sup\limits_{|x|\le ct,\  y\in[0,R]}|(u(t,x),v(t,x,y))-(U_R(x),V_R(x,y))|=0.
  	 \end{equation*}
  	 \end{enumerate}
  \end{theorem}

 Before stating the result of pulsating fronts for \eqref{pb in strip}, let us give the definition of pulsating traveling fronts in the strip $\overline\Omega_R$ for clarity.
\begin{definition}\label{Def2.3}
	A  rightward pulsating front of \eqref{pb in strip} connecting  $(U_R(x), V_R(x,y))$ to $(0,0)$  with effective mean speed $c\in\mathbb{R}_+$ is a time-global classical solution  $(u(t,x),v(t,x,y))=(\phi_R(x-ct,x),\psi_R(x-ct,x,y))$ of \eqref{pb in strip} such that the following periodicity property holds:
	\begin{equation}
	\label{periodicity-strip}
	u(t+\frac{k}{c},x)=u(t,x-k),\quad  v(t+\frac{k}{c},x,y)=v(t,x-k,y) \quad \forall k\in L\mathbb{Z},~ \forall t\in\mathbb{R}, ~\forall (x,y)\in\overline\Omega_R.
	\end{equation}
	Moreover, the profile $(\phi_R(s,x),\psi_R(s,x,y))$ satisfies
	\begin{equation}
	\label{limit-strip}
	\begin{cases}
	\phi_R(-\infty,x)=U_R(x), \ \phi_R(+\infty,x)=0 \ \text{uniformly in} \ x\in\mathbb{R},\\
	\psi_R(-\infty,x,y)=V_R(x,y), \ \psi_R(+\infty,x,y)=0  \ \text{uniformly in} \ (x,y)\in\overline\Omega_R,
	\end{cases}
	\end{equation}	
	with $(\phi_R(s,x), \psi_R(s,x,y))$ being continuous in $s\in\mathbb{R}$.	
	
	Similarly,  a  leftward pulsating front of \eqref{pb in strip}  connecting $(0,0)$ to  $(U_R(x), V_R(x,y))$ with effective mean speed $c\in\mathbb{R}_+$ is a time-global classical solution  $(\tilde u(t,x),\tilde v(t,x,y))=(\phi_R(x+ct,x),\psi_R(x+ct,x,y))$ of \eqref{pb in strip} such that the following periodicity property holds:
	\begin{equation*}
	\tilde u(t+\frac{k}{c},x)=\tilde u(t,x+k),\quad  \tilde v(t+\frac{k}{c},x,y)=\tilde v(t,x+k,y) \quad \forall k\in L\mathbb{Z},~\forall t\in\mathbb{R}, ~\forall (x,y)\in\overline\Omega_R.
	\end{equation*}
	Moreover, the profile $(\phi_R(s,x),\psi_R(s,x,y))$ satisfies
	\begin{equation*}
	\begin{cases}
	\phi_R(-\infty,x)=0, \ \phi_R(+\infty,x)=U_R(x) \ \text{uniformly in} \ x\in\mathbb{R},\\
	\psi_R(-\infty,x,y)=0, \ \psi_R(+\infty,x,y)=V_R(x,y)  \ \text{uniformly in} \ (x,y)\in\overline\Omega_R,
	\end{cases}
	\end{equation*}	
	with $(\phi_R(s,x), \psi_R(s,x,y))$ being continuous in $s\in\mathbb{R}$.	
\end{definition}

\begin{theorem}\label{thm-PTF-in strip}
	Let $c^*_{R}$ be given as in Theorem \ref{thm-asp-strip}. Then
	the following statements are valid:
	\begin{enumerate}[(i)]
		\item Problem \eqref{pb in strip} admits a rightward pulsating front  connecting $(U_R(x), V_R(x,y))$ to $(0,0)$ with wave profile $(\phi_R(s,x),\psi_R(s,x,y))$  being continuous and decreasing in $s$ if and only if $c\ge c^*_{R}$.
		\item  Problem \eqref{pb in strip} admits a leftward pulsating front  connecting $(0,0)$ to $(U_R(x), V_R(x,y))$  with wave profile $(\phi_R(s,x),\psi_R(s,x,y))$  being continuous and increasing in $s$ if and only if $c\ge c^*_{R}$.
	\end{enumerate}
\end{theorem}

 Theorems \ref{thm-asp-strip} and  \ref{thm-PTF-in strip} are proved simultaneously. It is worth mentioning that, compared with   the homogeneous field-road model  \cite{BRR2013-1} where there exists a unique minimal speed $w^*$ along the road in the left and right directions, here we get a striking resemblance. That is,    with the spatially periodic assumption and one-dimentional  setting on the road, the KPP minimal wave speeds in the right and left directions  are still the same,  even if there is no	homogeneity in $x$-direction anymore. 
However, the asymptotic spreading speeds may differ  in general, according to the direction in heterogeneous media and/or in higher dimension.

Having the  principal eigenvalue $\lambda_R(\alpha)$ for eigenvalue problem \eqref{eigen-pb-strip} in hand, we  construct in Section \ref{generalized p-e}  the \textit{generalized} principal eigenvalue $\lambda(\alpha)$ by passing to the limit $R\to+\infty$ in $\lambda_R(\alpha)$ for each $\alpha\in\mathbb{R}$, and show that there exists a  positive and $L$-periodic  (in $x$) solution $(P_\alpha,Q_\alpha)$ of the following generalized eigenvalue problem in the half-plane: 
\begin{align}
\label{generalized ep}
\begin{cases}
-D P_\alpha''+2D\alpha P_\alpha'+(-D\alpha^2+\mu) P_\alpha-\nu Q_\alpha(x,0)=\lambda(\alpha) P_\alpha,  &x\in\mathbb{R},\\
-d\Delta Q_\alpha+2d\alpha\partial_x Q_\alpha-(d\alpha^2+f_v(x,0)) Q_\alpha=\lambda(\alpha) Q_\alpha,  &(x,y)\in\Omega,\\
-d\partial_y Q_\alpha(x,0)+\nu Q_\alpha(x,0)-\mu P_\alpha=0, \   &x\in\mathbb{R},\\
 P_\alpha,  Q_\alpha \ \text{are} \ L\text{-periodic with respect to} \ x.
\end{cases}
\end{align}
We call  $(P_\alpha,Q_\alpha)$  the  generalized principal eigenfunction pair associated with $\lambda(\alpha)$.

We are now in a position to give the spreading result in the half-plane.

 \begin{theorem}\label{thm-asp-half-plane} 
	Let $(u, v)$ be the solution of  \eqref{pb in half-plane} with a nonnegative, bounded and continuous initial datum $(u_0,v_0)\not\equiv (0,0)$. Then there exists $c^*>0$ defined by
	\begin{equation*}
	c^*=\inf\limits_{\alpha>0}\frac{-\lambda(\alpha)}{\alpha},
	\end{equation*}
called the asymptotic spreading speed,  such that the following statements are valid:	
	\begin{itemize}	
		\item[(i)] If $(u_0,v_0)$ is compactly supported, then for any $c> c^*$, for any $A>0$,
		\begin{equation*}
		\lim\limits_{ t\to+\infty}\sup\limits_{|x|\ge ct,\ 0\le y\le A} |(u(t,x),v(t,x,y))|=0.
		\end{equation*}
		\item[(ii)] If $(u_0,v_0)<(\nu/\mu,1)$, then, for any $0<c<c^*$, for any $A>0$, 
		\begin{equation}
		\label{find lower bound}
		\lim\limits_{ t\to+\infty}\sup\limits_{|x|\le ct,\ 0\le y\le A} |(u(t,x),v(t,x,y))-(\nu/\mu,1)|=0.
		\end{equation}
	\end{itemize}
\end{theorem}

In the proof of Theorem \ref{thm-asp-half-plane}, the generalized principal  eigenfunction pair  $(P_\alpha,Q_\alpha)$ of \eqref{generalized ep} associated with $\lambda(\alpha)$  will play a crucial role in getting the upper bound for the spreading result. The lower bound follows from spreading results in the truncated domain via an asymptotic method.

Next, we state the concept of pulsating  fronts for problem \eqref{pb in half-plane} in the half-plane $\Omega$.

\begin{definition}\label{Def2.6}
	A rightward pulsating front of \eqref{pb in half-plane} connecting $(\nu/\mu,1)$ and $(0,0)$ with effective mean speed $c\in\mathbb{R}_+$ is a time-global classical solution $(u(t,x),v(t,x,y))=(\phi(x-ct,x),\psi(x-ct,x,y))$ of \eqref{pb in half-plane} such that the following periodicity property holds:
	\begin{equation*}
	u(t+\frac{k}{c},x)=u(t,x-k),\quad  v(t+\frac{k}{c},x,y)=v(t,x-k,y) \quad \forall k\in L\mathbb{Z},~ \forall t\in\mathbb{R},~\forall (x,y)\in\overline{\Omega}.
	\end{equation*}
	Moreover, the profile $(\phi(s,x),\psi(s,x,y))$ satisfies
\begin{align}
\label{PTF-limit condition}
\begin{cases}
\phi(-\infty,x)={\nu}/{\mu}, \ \phi(+\infty,x)=0 \ \text{uniformly in} \ x\in\mathbb{R},\\
\psi(-\infty,x,y)=1, \ \psi(+\infty,x,y)=0  \ \text{uniformly in} \ x\in\mathbb{R},  \ \text{locally uniformly in} \ y\in [0,+\infty),
\end{cases}
\end{align} 		
with $(\phi(s,x), \phi(s,x,y))$ being continuous in $s\in\mathbb{R}$.

	Similarly, a leftward pulsating front of \eqref{pb in half-plane} connecting  $(0,0)$ and $(\nu/\mu,1)$ with  effective mean speed $c\in\mathbb{R}_+$ is a time-global classical solution $(u(t,x),v(t,x,y))=(\phi(x+ct,x),\psi(x+ct,x,y))$ of \eqref{pb in half-plane} such that the following periodicity property holds:
\begin{equation*}
\tilde u(t+\frac{k}{c},x)=\tilde u(t,x+k),\quad  \tilde v(t+\frac{k}{c},x,y)=\tilde v(t,x+k,y) \quad \forall k\in L\mathbb{Z},~ \forall t\in\mathbb{R},~\forall (x,y)\in\overline{\Omega}.
\end{equation*}
Moreover, the profile $(\phi(s,x),\psi(s,x,y))$ satisfies
\begin{align*}
\begin{cases}
\phi(-\infty,x)={\nu}/{\mu}, \ \phi(+\infty,x)=0 \ \text{uniformly in} \ x\in\mathbb{R},\\
\psi(-\infty,x,y)=1, \ \psi(+\infty,x,y)=0  \ \text{uniformly in} \ x\in\mathbb{R},  \ \text{locally uniformly in} \ y\in [0,+\infty),
\end{cases}
\end{align*} 		
with $(\phi(s,x), \phi(s,x,y))$ being continuous in $s\in\mathbb{R}$.
\end{definition}

Based on Theorem \ref{thm-PTF-in strip} and an asymptotic method, we can show:

\begin{theorem}\label{PTF-in half-plane}
	Let $c^*$ be defined as in Theorem \ref{thm-asp-half-plane}. Then
	the following statements are valid:
	\begin{enumerate}[(i)]
		\item Problem \eqref{pb in half-plane} admits a rightward pulsating front  connecting $(\nu/\mu, 1)$ to $(0,0)$ with wave profile $(\phi(s,x),\psi(s,x,y))$  being continuous and decreasing in $s$ if and only if $c\ge c^*$.
		\item  Problem \eqref{pb in half-plane} admits a leftward pulsating front  connecting $(0,0)$ to $(\nu/\mu, 1)$  with wave profile $(\phi(s,x),\psi(s,x,y))$  being continuous and increasing in $s$ if and only if $c\ge c^*$.
	\end{enumerate}
\end{theorem}

{\bf Outline of the paper.} The remaining part of this paper is organized as follows. In Section \ref{section3}, we state some preliminary results for problem \eqref{pb in half-plane} in the half-plane and for problem \eqref{pb in strip} in the strip, respectively. We prove  Liouville-type results and investigate large time behaviors for problem \eqref{pb in half-plane} in Section 3.1 and for problem \eqref{pb in strip} in Section 3.2, respectively. 
Section \ref{section4} is dedicated to the proofs of Theorems \ref{thm-asp-strip} and  \ref{thm-PTF-in strip}. Particularly, the principal eigenvalue problem in $\overline\Omega_R$ is investigated, see Proposition \ref{principal eigenvalue in strip}. In Section \ref{section5}, we give the proofs of Theorems \ref{thm-asp-half-plane} and  \ref{PTF-in half-plane},  based on the study of the generalized principal eigenvalue problem and the results derived for truncated problems. 
\section{Preliminaries}\label{section3}

In this section, we state some auxiliary results in the half-plane and in the truncated domain, respectively. Specifically,  two versions of  comparison principles that will be diffusely used throughout this paper and the well-posedness of the Cauchy problems   for problem \eqref{pb in half-plane} in the half-line and for problem \eqref{pb in strip} in the strip will be given below, respectively. Since the results  can be proved by slight modifications of the arguments in \cite{BRR2013-1}, we omit the proofs here. We also prove Liouville-type results and large time behavior of solutions to Cauchy problems  \eqref{pb in half-plane} and \eqref{pb in strip}, respectively. Finally,
we investigate the limiting property of the stationary solution in the strip when the width of the strip goes to infinity.

 In the sequel, a subsolution (resp. supersolution) is a couple satisfying the system in the classical sense with $=$ replaced by $\le$ (resp. $\ge$) which is continuous up to time 0.  We say that a function is a generalized supersolution (resp. subsolution) if it is the minimum (resp. maximum) of a finite number of supersolutions (resp. subsolutions).

\subsection{Preliminary results in the half-plane}

\begin{proposition}
	\label{cp}
	Let $(\underline{u},\underline{v})$ and $(\overline{u}, \overline{v})$ be respectively a subsolution bounded from above and a supersolution bounded from below of \eqref{pb in half-plane} satisfying $\underline{u}\le\overline{u}$ and $\underline{v}\le\overline{v}$ in $\overline\Omega$ at $t=0$. Then, either $\underline{u}<\overline{u}$ and $\underline{v}<\overline{v}$ in $\overline\Omega$ for all $t>0$, or there exists $T>0$ such that $(\underline{u},\underline{v})=(\overline{u}, \overline{v})$ in $\overline\Omega$ for $t\le T$.
\end{proposition}

\begin{proposition}
	\label{cp--generalized sub}
	Let $E\subset (0,+\infty)\times\mathbb{R}$ and $F\subset (0,+\infty)\times\Omega$ be two open sets and let $(u_1, v_1)$ and $(u_2,v_2)$ be two subsolutions of \eqref{pb in half-plane} bounded from above, satisfying 
	\begin{equation*}
	u_1\le u_2 \ \text{on} \ (\partial E)\cap((0,+\infty)\times\mathbb{R}),\qquad v_1\le v_2\ \text{on} \ (\partial F)\cap((0,+\infty)\times\Omega).
	\end{equation*} 
	If the functions $\underline{u}$, $\underline{v}$ defined by
	\begin{align*}
	\underline{u}(t,x)&:=
	\begin{cases}
	\max\{u_1(t,x), u_2(t,x)\}\qquad &\text{if}\ (t,x)\in\overline{E}\\
	u_2(t,x)\qquad&\text{otherwise}
	\end{cases}\\
	\underline{v}(t,x,y)&:=
	\begin{cases}
	\max\{v_1(t,x,y), v_2(t,x,y)\}~ &\text{if}\ (t,x,y)\in\overline{F}\\
	v_2(t,x,y)\qquad&\text{otherwise}
	\end{cases}
	\end{align*}
	satisfy
	\begin{align*}
	\underline{u}(t,x)>u_2(t,x)\Rightarrow \underline{v}(t,x,0)\ge v_1(t,x,0),\\
		\underline{v}(t,x,0)>v_2(t,x,0)\Rightarrow \underline{u}(t,x)\ge u_1(t,x),
	\end{align*}
	then, any supersolution $(\overline{u},\overline{v})$ of \eqref{pb in half-plane} bounded from below and such that $\underline{u}\le\overline{u}$ and $\underline{v}\le\overline{v}$ in $\overline\Omega$ at $t=0$, satisfies $\underline{u}\le\overline{u}$ and $\underline{v}\le\overline{v}$ in $\overline\Omega$ for all $t>0$.
\end{proposition}


\begin{proposition}
	\label{wellposedness-half-plane}
	The Cauchy problem \eqref{pb in half-plane} with nonnegative,  bounded and continuous initial datum $(u_0,v_0)\not\equiv(0,0)$ admits a unique nonnegative classical bounded solution $(u,v)$ for $t\ge 0$ and $(x,y)\in\overline\Omega$.
\end{proposition}




Now we  prove a Liouville-type result for the stationary problem corresponding to \eqref{pb in half-plane} as well as the long time behavior of solutions for the Cauchy problem \eqref{pb in half-plane}.
\begin{theorem}
	\label{liouville}
Problem \eqref{pb in half-plane} has a unique positive bounded stationary solution $(U,V)\equiv({\nu}/{\mu}, 1)$. Moreover, let $(u,v)$ be the solution of \eqref{pb in half-plane} with a nonnegative,  bounded and continuous initial datum $(u_0, v_0)\not\equiv(0,0)$, then 
 	\begin{equation}
 	\label{large time}
 	\lim\limits_{t\to +\infty}(u(t,x),v(t,x,y))= (\nu/\mu, 1)\quad\text{locally uniformly for}\ (x,y)\in\overline{\Omega}.
  	\end{equation}
\end{theorem}
\begin{proof}
	Let $(u,v)$ be the  solution, given in Proposition \ref{wellposedness-half-plane}, of the Cauchy problem \eqref{pb in half-plane} starting from a nonnegative, bounded  and  continuous initial datum $(u_0, v_0)\not\equiv(0,0)$.		
	We first show the existence of the nontrivial nonnegative and bounded stationary solution of \eqref{pb in half-plane}, by using a sub- and supersolution argument. Take $\rho>0$ large enough such that the principal eigenvalue of $-\Delta$ in $B_\rho\subset\mathbb{R}^2$ with Dirichlet boundary condition is less than $m/(2d)$ (recall that $m=\min_{[0,L]}f_v(x,0)>0$). Then, the associated principal eigenfunction $\varphi_\rho$ satisfies
$-d\Delta \varphi_\rho\le m\varphi_\rho/2$ in $B_\rho$.
	Hence, there exists $\varep_0>0$ such that the function $\varep\varphi_\rho$ satisfies $-d\Delta(\varep\varphi_\rho)\le f(x,\varep\varphi_\rho)$ in $B_\rho$  for all $\varep\in (0,\varep_0]$. Define $\underline V(x,y)=\varep\varphi_\rho(x,y-\rho-1)$ in $B_\rho(0,\rho+1)$ and extend it by 0 outside.
	The function pair $(0,\underline{V})$ is nonnegative, bounded and continuous. On the other hand, Proposition \ref{cp} implies that $(u,v)$ is positive for $t>0$ and $(x,y)\in\overline\Omega$. Hence, up to decreasing $\varep$ if necessary, we have that $(0, \underline{V})$ is below $(u(1,\cdot),v(1,\cdot,\cdot))$ in $\overline\Omega$. Let $(\underline{u}, \underline{v})$ be the unique bounded classical solution of \eqref{pb in half-plane} with initial condition $(0, \underline{V})$.  It then follows from Proposition \ref{cp--generalized sub} that $(\underline{u},\underline{v})$ is  increasing in $t$ and $(\underline u(t,x),\underline v(t,x,y))<(u(t+1,x),v(t+1,x,y))$ for $t>0$ and $(x,y)\in\overline\Omega$.  
	By parabolic estimates, $(\underline{u},\underline{v})$ converges locally uniformly in $\overline\Omega$ as $t\to+\infty$ to a stationary solution $(U_1, V_1)$ of \eqref{pb in half-plane} satisfying 
	\begin{equation}
	\label{lower-0}
	0<U_1\le\liminf\limits_{ t\to+\infty}u(t,x),\qquad \underline{V}<V_1\le\liminf\limits_{ t\to+\infty}v(t,x,y).
	\end{equation}

	On the other hand, define 
	\begin{equation}\label{super-constant sol}
	(\overline{U},\overline{V}):=\max\bigg\{\frac{\Vert u_0\Vert_{L^{\infty}(\mathbb{R})}}{\nu},\frac{\Vert v_0\Vert_{L^{\infty}(\overline\Omega)}+1}{\mu}\bigg\}(\nu,\mu).
	\end{equation}
	Obviously, $(\overline{U},\overline{V})$ is a supersolution of \eqref{pb in half-plane} and satisfies $(\overline{U},\overline{V})\ge (u_0,v_0)$. Let $(\overline{u},\overline{v})$ be the solution of \eqref{pb in half-plane} with initial datum $(\overline{U}, \overline{V})$, then Proposition \ref{cp} implies that $(\overline{u},\overline{v})$ is non-increasing in $t$. From parabolic estimates, $(\overline{u},\overline{v})$ converges locally uniformly in $\overline\Omega$ as $t\to+\infty$ to a stationary solution $(U_2, V_2)$ of \eqref{pb in half-plane} satisfying 
	\begin{equation}
	\label{upper-0}
	\limsup\limits_{ t\to+\infty}u(t,x)\le U_2\le \overline U,\qquad \limsup\limits_{ t\to+\infty}v(t,x,y)\le V_2\le \overline V.
	\end{equation}
 Therefore, the existence of nontrivial nonnegative and bounded stationary solutions of \eqref{pb in half-plane} is proved.

 Let $(U,V)$ be an arbitrary nontrivial nonnegative and bounded stationary solution of \eqref{pb in half-plane}. 
 In the spirit of \cite[Proposition 4.1]{BRR2013-1} and \cite[Lemma 2.5]{GMZ2015}, one can further conclude that  $\inf_{\overline\Omega}V>0$ and then $\inf_{\mathbb{R}}U>0$, by using the hypothesis $m:=\min_{[0,L]}f_v(x,0)>0$.

 Next, we show the uniqueness. 
Assume that $(U_1, V_1)$ and $(U_2, V_2)$ are  two distinct positive bounded stationary solutions of \eqref{pb in half-plane}. Then, there is $\varep>0$ such that $U_i\ge\varep$ in $\mathbb{R}$ and $V_i\ge\varep$ in $\overline\Omega$ for $i=1,2$. Therefore, we can define
	\begin{equation*}
	\theta^*:=\sup\big\{\theta>0:~ (U_1, V_1)>\theta(U_2, V_2)~~ \text{in} ~\overline\Omega\big\}>0.
	\end{equation*} 
Assume that $\theta^*<1$. Set
$P:=U_1-\theta^*U_2\ge 0$ in $\mathbb{R}$ and $Q:=V_1-\theta^*V_2\ge 0$ in $\overline\Omega$.
From the definition of $\theta^*$, there exists a sequence $(x_n,y_n)_{n\in\mathbb{N}}$ in $\overline\Omega$ such that
	$P(x_n)\to 0$, or $Q(x_n,y_n)\to 0$ as $n\to+\infty$. 
	
Assume that the second case occurs, we claim that $(y_n)_{n\in\mathbb{N}}$ is bounded. Assume by contradiction that $y_n\to+\infty$ as $n\to+\infty$, then, up to extraction of a subsequence, the functions $V_{i,n}(x,y):=V_i(x,y+y_n)$ $(i=1,2)$ converge locally uniformly to positive bounded functions $V_{i,\infty}$ solving  $-d\Delta V_{i,\infty}=f(x,V_{i,\infty})$ in $\mathbb{R}^2$, which implies that $V_{i,\infty}\equiv 1$ in $\mathbb{R}^2$,  because of the KPP hypothesis on $f$. Then, it follows that $Q(x_n,y_n)\to 1-\theta^*>0$ as $n\to+\infty$, which is a contradiction. Thus, the sequence $(y_n)_{n\in\mathbb{N}}$ must be bounded. We then distinguish two subcases.
 
	Assume  that, up to a subsequence, $x_n\to\overline x\in\mathbb{R}$  and $y_n\to \overline y\in[0,+\infty)$ as $n\to+\infty$. By continuity, one has $Q\ge 0$ in $\overline\Omega$ and $Q(\overline x, \overline y)=0$. Suppose that $\overline y>0$. Notice that $Q$ satisfies
		\begin{equation}
		\label{Q-1}
		-d\Delta Q=f(x,V_1)-\theta^*f(x,V_2)~~\text{in}~\Omega,
		\end{equation}
	Since  $f(x,v)/v$ is  decreasing in $v>0$ for all $x\in\mathbb{R}$ and since $\theta^*<1$, it follows that 
	$-d\Delta Q>f(x,V_1)-f(x,\theta^*V_2)$ in $\Omega$.
 Since $f$ is locally Lipschitz continuous in the second variable,  there exists a bounded function $b(x,y)$ defined in $\Omega$ such that
\begin{equation}	\label{Q-eq-half-plane}
	-d\Delta Q+b Q>0~~\text{in}~\Omega.
	\end{equation}
Since $Q\ge 0$ in $\overline\Omega$ and $Q(\overline x,\overline y)=0$, it follows from the strong maximum principle that $Q\equiv 0$ in $\Omega$, which contradicts the strict inequality in \eqref{Q-eq-half-plane}. Hence, $Q>0$ in $\Omega$. Suppose now that $\overline y=0$, then $Q(\overline x,0)=0$. The Hopf lemma implies that $\partial_y Q(\overline x,0)>0$. Using the boundary condition, one gets $0>-d\partial_y Q(\overline x,0)=\mu P(\overline x)-\nu Q(\overline x,0)=\mu P(\overline x)\ge 0$. This is a contradiction. Therefore, $Q>0$ in $\overline\Omega$.

In the general case, let  $\overline{x}_n\in [0,L]$ be such that $x_n-\overline{x}_n\in L\mathbb{Z}$, then up to extraction of a subsequence,  $\overline{x}_n\to \overline x_{\infty}\in[0,L]$ as $n\to\infty$.
	Since $y_n$ are bounded, up to extraction of a further subsequence, one gets $y_n\to \overline y_\infty\in[0,+\infty)$ as $n\to+\infty$. Let us set $U_{i,n}(x):=U_i(x+x_n)$ and $V_{i,n}(x,y):=V_i(x+x_n,y+y_n)$ for $i=1,2$. Then, $(U_{i,n}, V_{i,n})$ satisfies
	\begin{align}
	\label{U-V-1}
	\begin{cases}
	-DU''_{i,n}=\nu V_{i,n}(x,0)-\mu U_{i,n}~~&\text{in}~\mathbb{R},\cr
	-d\Delta V_{i,n}=f(x+\overline x_n, V_{i,n})~~&\text{in}~\Omega,\cr
	-d\partial_y V_{i,n}(x,0)=\mu U_{i,n}-\nu V_{i,n}(x,0)~~&\text{in}~\mathbb{R}.
	\end{cases}
	\end{align}
	From standard elliptic estimates, it follows that, up to a subsequence, $(U_{i,n}, V_{i,n})$ converges locally uniformly  as $n\to+\infty$  to a classical solution $(U_{i,\infty}, V_{i,\infty})$ of
		\begin{align}\label{U-V-2}
	\begin{cases}
	-DU''_{i,\infty}=\nu V_{i,\infty}(x,0)-\mu U_{i,\infty}~~&\text{in}~\mathbb{R},\cr
	-d\Delta V_{i,\infty}=f(x+\overline x_\infty, V_{i,\infty})~~&\text{in}~\Omega,\cr
	-d\partial_y V_{i,\infty}(x,0)=\mu U_{i,\infty}-\nu V_{i,\infty}(x,0)~~&\text{in}~\mathbb{R}.
	\end{cases}
	\end{align}
	Moreover, there is $\varep>0$ such that $U_{i,\infty}\ge \varep$ in $\mathbb{R}$ and $V_{i,\infty}\ge\varep$ in $\overline\Omega$.

	Set $P_\infty:=U_{1,\infty}-\theta^* U_{2,\infty}$ in $\mathbb{R}$, and $Q_\infty:=V_{1,\infty}-\theta^* V_{2,\infty}$ in $\overline\Omega$. Then, $P_\infty\ge 0$ in $\mathbb{R}$, $Q_\infty\ge 0$ in $\overline\Omega$ and $Q_\infty(\overline x_\infty,\overline y_\infty)=0$. Suppose that $\overline y_\infty>0$. Notice that $Q_\infty$ satisfies
		\begin{equation*}
		-d\Delta Q_\infty=f(x+\overline x_\infty, V_{1,\infty})-\theta^* f(x+\overline x_\infty, V_{2,\infty})~~~\text{in}~\Omega.
		\end{equation*}
	By analogy with the analysis for problem \eqref{Q-1}, one eventually obtains that $Q_\infty>0$ in $\Omega$. One can exclude the case that $\overline y_\infty=0$, by using again the Hopf lemma and the boundary condition. Therefore, $Q_\infty>0$ in $\overline\Omega$.
	Thus, the case that $Q(x_n,y_n)\to 0$ is ruled out.

	It is left to discuss the first case that $P(x_n)\to 0$ as $n\to+\infty$. Assume first that, up to a subsequence, $x_n\to\overline x$ as $n\to+\infty$. By continuity, one has $P\ge 0$ in $\mathbb{R}$ and $P(\overline x)=0$. Moreover, $P$ satisfies $-DP''+\mu P=\nu Q(\cdot,0)>0$ in $\mathbb{R}$. The strong maximum principle implies that $P\equiv 0$ in $\mathbb{R}$, which is a contradiction. In the general case, let  $\overline{x}_n\in [0,L]$ be such that $x_n-\overline{x}_n\in L\mathbb{Z}$, then up to a subsequence,  $\overline{x}_n\to \overline x_{\infty}\in[0,L]$ as $n\to\infty$. Set $U_{i,n}(x):=U_i(x+x_n)$ in $\mathbb{R}$ and $V_{i,n}(x,y):=V_i(x+x_n,y)$ in $\overline\Omega$, for $i=1,2$. Then  $(U_{i,n},V_{i,n})$ satisfies \eqref{U-V-1} in $\overline\Omega$. 	From standard elliptic estimates, it follows that, up to a subsequence, $U_{i,n}$ and $V_{i,n}$ converge as $n\to+\infty$ in $C^2_{loc}$ to $U_{i,\infty}$ and $V_{i,\infty}$, respectively, which satisfy \eqref{U-V-2}.
	Set $P_\infty:=U_{1,\infty}-\theta^* U_{2,\infty}$ in $\mathbb{R}$ and $Q_\infty:=V_{1,\infty}-\theta^* V_{2,\infty}$ in $\overline\Omega$. Then, $P_\infty\ge 0$ in $\mathbb{R}$ and $P_\infty(0)=0$, $Q_\infty> 0$ in $\overline\Omega$. Moreover, it holds 
	\begin{equation*}
	-DP''_\infty+\mu P_\infty=\nu Q_\infty(\cdot,0)>0~~\text{in}~\mathbb{R}.
	\end{equation*}
	The strong maximum principle then implies that $P_\infty\equiv 0$ in $\mathbb{R}$, which is a contradiction. Hence, the case that $P(x_n)\to 0$ is also impossible.

Consequently,  $\theta^*\ge 1$, whence $(U_1,V_1)\ge (U_2,V_2)$ in $\overline\Omega$. By interchanging the roles of $(U_1,V_1)$ and $(U_2,V_2)$, one can show that $(U_2,V_2)\ge (U_1,V_1)$ in $\overline\Omega$. The uniqueness is achieved.

 Furthermore, if $(U,V)$  is a positive bounded stationary solution of \eqref{pb in half-plane}, it is easy to check that any $L$-lattice translation in $x$ of $(U,V)$ is still a positive bounded stationary solution of \eqref{pb in half-plane}.  Thus, $(U,V)$ is $L$-periodic in $x$.
	It is straightforward to check that $(\nu/\mu,1)$ satisfies the stationary  problem of \eqref{pb in half-plane}. 
   Therefore, 
   $(\nu/\mu,1)$ is the unique positive and bounded stationary solution of \eqref{pb in half-plane}. The large time behavior \eqref{large time} of the solution to the Cauchy problem \eqref{pb in half-plane} then follows immediately from \eqref{lower-0} and \eqref{upper-0}. 
The proof of Theorem \ref{liouville} is thereby complete.
\end{proof}

\subsection{Preliminary results in the strip}
\label{section3-2}

\begin{proposition}
	\label{cp-strip}
	Let $(\underline{u},\underline{v})$ and $(\overline{u}, \overline{v})$ be respectively a subsolution bounded from above and a supersolution bounded from below of \eqref{pb in strip} satisfying $\underline{u}\le\overline{u}$ and $\underline{v}\le\overline{v}$ in $\overline\Omega_R$ at $t=0$. Then, either $\underline{u}<\overline{u}$ and $\underline{v}<\overline{v}$ in $\mathbb{R}\times[0,R)$ and $\partial_y\bar v(t,x,R)<\partial_y\underline v(t,x,R)$ on $\mathbb{R}$  for all $t>0$, or there exists $T>0$ such that $(\underline{u},\underline{v})=(\overline{u}, \overline{v})$ in $\overline\Omega_R$ for $t\le T$.
\end{proposition}
\begin{proposition}
	\label{cp--generalized sub-strip}
	Let $E\subset(0,+\infty)\times \mathbb{R}$ and $F\subset (0,+\infty)\times\Omega_R$ be two open sets and let $(u_1, v_1)$ and $(u_2,v_2)$ be two subsolutions of \eqref{pb in strip} bounded from above, satisfying 
	\begin{equation*}
	u_1\le u_2 \ \text{on} \ (\partial E)\cap((0,+\infty)\times\mathbb{R}),\qquad v_1\le v_2\ \text{on} \ (\partial F)\cap((0,+\infty)\times\Omega_R).
	\end{equation*} 
	If the functions $\underline{u}$, $\underline{v}$ defined by
	\begin{align*}
	\underline{u}(t,x)&:=
	\begin{cases}
	\max\{u_1(t,x), u_2(t,x)\}\qquad &\text{if}\ (t,x)\in\overline{E}\\
	u_2(t,x)\qquad\qquad\qquad\qquad&\text{otherwise}
	\end{cases}\\
	\underline{v}(t,x,y)&:=
	\begin{cases}
	\max\{v_1(t,x,y), v_2(t,x,y)\}~ &\text{if}\ (t,x,y)\in\overline{F}\\
	v_2(t,x,y)\qquad\qquad&\text{otherwise}
	\end{cases}
	\end{align*}
	satisfy
	\vspace{-3mm}
	\begin{align*}
	\underline{u}(t,x)>u_2(t,x)\Rightarrow \underline{v}(t,x,0)\ge v_1(t,x,0),\\
	\underline{v}(t,x,0)>v_2(t,x,0)\Rightarrow \underline{u}(t,x)\ge u_1(t,x),
	\end{align*}
	then, any supersolution $(\overline{u},\overline{v})$ of \eqref{pb in strip} bounded from below and such that $\underline{u}\le\overline{u}$ and $\underline{v}\le\overline{v}$ in $\overline\Omega_R$ at $t=0$, satisfies $\underline{u}\le\overline{u}$ and $\underline{v}\le\overline{v}$ in $\overline\Omega_R$ for all $t>0$.
\end{proposition}
\begin{proposition}
	\label{wellposedness-strip}
	The Cauchy problem \eqref{pb in strip} with nonnegative,  bounded and  continuous initial datum $(u_0,v_0)\not\equiv (0,0)$ admits a unique bounded classical solution $(u,v)$ for all $t\ge 0$ and $(x,y)\in\overline\Omega_R$. Moreover, for any $0<\tau<T$ and	for any compact subsets $I\subset\mathbb{R}$ and $H\subset\Omega_R$ with $\overline H\cap\{y=0\}=\overline I$, 
	\begin{align*}
	\Vert u(t,x)\Vert_{C^{1+\frac{\alpha}{2},2+\alpha}([\tau,T]\times \overline I)}+\Vert v(t,x,y)\Vert_{C^{1+\frac{\alpha}{2},2+\alpha}([\tau,T]\times \overline H)}\le C,
	\end{align*}
	where $C$ is a positive constant depending on $\tau$,$ T$, $f$, $\Vert u_0\Vert_{L^\infty(\mathbb{R})}$ and $\Vert v_0\Vert_{L^\infty(\overline\Omega_R)}$.
\end{proposition}

  	The existence of the solution to the Cauchy problem \eqref{pb in strip} follows from an approximation argument by constructing a sequence of approximate solutions in $[-n,n]\times[0,R]$ for $n$ large enough, which satisfy\footnote{Problem \eqref{pb in strip-cut off} with a nonnegative, continuous and bounded initial function  has a unique bounded classical   solution defined for all $t>0$, which can be obtained in the spirit of  Appendix A in \cite{BRR2013-1} and the strong maximum principle.}
  \begin{equation}
  \label{pb in strip-cut off}
  \begin{cases}
  \partial_t u -D\partial_{xx}u=\nu v(t,x,0)-\mu u, &t>0,\  x\in [-n,n], \\
  \partial_t v -d\Delta v =f(x,v), &t>0,\ (x,y)\in(-n,n)\times(0,R),\\
  -d\partial_y v(t,x,0)=\mu u-\nu v(t,x,0), &t>0,\  x\in [-n,n],\\
  v(t,x,R)=0, &t>0,\  x\in [-n,n],\\
  v(t,\pm n, y)=0, &t>0,\ y\in[0,R],
  \end{cases}
  \end{equation}
   and then passing to the limit $n\to+\infty$ via the Arzel\`{a}-Ascoli theorem. Uniqueness comes from the comparison principle Proposition \ref{cp-strip}.  The estimate can be derived by standard parabolic $L^p$ theory (see, e.g., \cite[page 168, Proposition 7.14]{Liberman}) and then the Schauder theory. 
  	
  	In the following, we show the continuous dependence of the solutions to the Cauchy problem \eqref{pb in strip} on initial data.

\begin{proposition}
	\label{prop3.8}
	The solutions of the Cauchy problem \eqref{pb in strip} depend continuously on the initial data. 	
\end{proposition}

\begin{proof}
  Let $(u,v)$ be the solution, given in Proposition \ref{wellposedness-strip}, of  \eqref{pb in strip} with nonnegative, bounded and continuous initial datum $(u_0,v_0)\not\equiv (0,0)$. We shall prove that for any $\varep>0$, $T>0$, there is $\delta>0$, depending on $\varep$, $T$ and $(u_0,v_0)$, such that for any nonnegative, bounded and continuous function pair $(\tilde u_0,\tilde v_0)$ satisfying 
	\begin{equation}
	\label{initial-1}
	\sup_{x\in\mathbb{R}}|u_0(x)-\tilde u_0(x)|<\frac{\nu}{\mu}\delta,~~\sup_{(x,y)\in\overline\Omega_R}|v_0(x,y)-\tilde v_0(x,y)|<\delta,
	\end{equation}
	the solution to \eqref{pb in strip} with initial value $(\tilde u_0,\tilde v_0)$ satisfies
	\begin{equation}\label{sol-1}
	\sup_{(t,x)\in[0,T]\times\mathbb{R}}|u(t,x)-\tilde u(t,x)|<\frac{\nu}{\mu}\varep,~~\sup_{(t,x,y)\in[0,T]\times\overline\Omega_R}|v(t,x,y)-\tilde v(t,x,y)|<\varep.
	\end{equation}
	
	Recall that $M=\max_{[0,L]}f_v(x,0)$. Define $(w,z):=(u,v)e^{-Mt}$, then $(w,z)$ satisfies
	\begin{equation}
	\label{wz}
	\begin{cases}
	\partial_t w -D\partial_{xx}w=\nu z(t,x,0)-(\mu+M) w, &t>0,\  x\in \mathbb{R}, \\
	\partial_t z -d\Delta z =g(t,x,z), &t>0,\ (x,y)\in\Omega_R,\\
	-d\partial_y z(t,x,0)=\mu w-\nu z(t,x,0), &t>0,\  x\in \mathbb{R},\\
	z(t,x,R)=0, &t>0,\  x\in \mathbb{R},
	\end{cases}
	\end{equation}
	where the function $g(t,x,z):=f(x,e^{Mt}z)e^{-Mt}-Mz$ is non-increasing in $z$.
	We observe that $(u,v)e^{-Mt}$ and $(\tilde u,\tilde v)e^{-Mt}$ are the solutions of \eqref{wz} with initial functions $(u_0,v_0)$ and $(\tilde u_0,\tilde v_0)$, respectively.
	
	Define 
	\begin{equation*}
	\begin{cases}
	\underline u(t,x):= \max\big(0,w(t,x)-\frac{\nu}{\mu}\delta(\frac{t}{T}+1)\big),~~~~~~\cr
	\underline v(t,x,y):=\max\big(0,z(t,x,y)-\delta(\frac{t}{T}+1)\big),
	\end{cases}
	\end{equation*}
	and 
	\begin{equation*}
	\begin{cases}
	\overline u(t,x):=\min\big(\frac{\nu}{\mu+M}A, w(t,x)+\frac{\nu}{\mu}\delta(\frac{t}{T}+1)\big),\cr
	\overline v(t,x,y):=\min\big(A, z(t,x,y)+\delta(\frac{t}{T}+1)\big),
	\end{cases}
	\end{equation*}
	where $A:=\max\Big(1,\Vert u_0 \Vert_{L^{\infty}(\mathbb{R})}+\Vert v_0 \Vert_{L^\infty(\overline\Omega_R)}+\delta, \frac{\mu+M}{\nu}(\Vert u_0 \Vert_{L^{\infty}(\mathbb{R})}+\Vert v_0 \Vert_{L^\infty(\overline\Omega_R)}+\frac{\nu}{\mu}\delta)\Big)$. It can be checked that $(\underline u,\underline v)$ and $(\overline u,\overline v)$ are, respectively, a generalized sub- and a generalized supersolution of \eqref{wz}. Notice that
	\begin{align*}
	\underline u(0,x)=\max\Big(0, u_0(x)-\frac{\nu}{\mu}\delta \Big)<&u_0(x)<u_0(x)+\frac{\nu}{\mu}\delta= \overline u(0,x),~~\forall x\in\mathbb{R},\cr
	\underline v(0,x,y)=\max\Big( 0, v_0(x,y)-\delta\Big)<&v_0(x,y)<v_0(x,y)+\delta= \overline v(0,x,y),~~\forall (x,y)\in\overline\Omega_R.
	\end{align*}
    From \eqref{initial-1}, one infers that
    \begin{align*}
    \underline u(0,x)=\max\Big(0, u_0(x)-\frac{\nu}{\mu}\delta \Big)<&\tilde u_0(x)<u_0(x)+\frac{\nu}{\mu}\delta= \overline u(0,x),~~\forall x\in\mathbb{R},\cr
    \underline v(0,x,y)=\max\Big( 0, v_0(x,y)-\delta\Big)<&\tilde v_0(x,y)<v_0(x,y)+\delta= \overline v(0,x,y),~~\forall (x,y)\in\overline\Omega_R.
    \end{align*}	
	By a comparison argument, it follows that
	\begin{equation*}
	(\underline u,\underline v)\le (u,v)e^{-Mt}\le (\overline u,\overline v),~~~~	(\underline u,\underline v)\le (\tilde u,\tilde v)e^{-Mt}\le (\overline u,\overline v),
	\end{equation*}
	for all $t\in[0,T]$ and $(x,y)\in\overline\Omega_R$. Thus, 
	\begin{equation*}
	\sup_{[0,T]\times\mathbb{R}}|u(t,x)-\tilde u(t,x)|\le e^{MT}\sup_{[0,T]\times\mathbb{R}}|\overline u(t,x)-\underline u(t,x)|\le 2e^{MT}\frac{\nu}{\mu}\delta\sup_{[0,T]}(\frac{t}{T}+1)\le 4e^{MT}\frac{\nu}{\mu}\delta,
	\end{equation*}
	\begin{equation*}
	\sup_{[0,T]\times\overline\Omega_R}|v(t,x,y)-\tilde v(t,x,y)|\le e^{MT}\sup_{[0,T]\times\overline\Omega_R}|\overline v(t,x,y)-\underline v(t,x,y)|\le 2e^{MT}\delta\sup_{[0,T]}(\frac{t}{T}+1)\le 4e^{MT}\delta.
	\end{equation*}
	By choosing $\delta>0$ so small that $4e^{MT}\delta<\varep$, \eqref{sol-1} is therefore achieved.	
\end{proof}

Next, we  prove a Liouville-type result in the strip, provided that the width $R$ is sufficiently large. Namely, for all $R$ large,  problem  \eqref{pb in strip} admits a unique nonnegative nontrivial stationary solution $(U_R,V_R)$, which is $L$-periodic in $x$. Moreover, we show that $(U_R,V_R)$ is the global attractor  for solutions of the Cauchy problem in the strip. 

\begin{proof}[Proof of Theorem \ref{thm2.1}]
 The strategy of this proof is similar in spirit to that of Theorem \ref{liouville}.   We only sketch the proof of the existence and  positivity property of stationary solutions, for which the construction of subsolutions is inspired by \cite[Proposition 3.4]{Tellini2016}.

 	Let $(u,v)$ be the solution to the Cauchy problem \eqref{pb in strip} with nonnegative, bounded and continuous initial value $(u_0,v_0)\not\equiv(0,0)$. Set
\begin{align*}
(\underline{u},\underline{v}):=\begin{cases}
\cos(\omega x)\Big(1, \frac{\mu\sin(\beta(R-y))}{d\beta\cos(\beta R)+\nu\sin(\beta R)}\Big)  \ &\text{for}\ (x,y)\in (-\frac{\pi}{2\omega},\frac{\pi}{2\omega})\times[0,R],\\
(0,0) &\text{otherwise,}
\end{cases}
\end{align*}
where $\omega$ and $\beta$ are parameters to be chosen later so that $(\underline{u},\underline{v})$ satisfies
\begin{equation}
\label{beta,w}
\begin{cases}
-D\underline{u}''\le \nu\underline{v}(x,0)-\mu \underline{u}, & x\in \mathbb{R}, \\
-d\Delta\underline{v} \le \big(m-\delta\big)\underline{v},  &(x,y)\in\Omega_R,\\
-d\partial_y \underline{v}(x,0)=\mu \underline{u}-\nu \underline{v}(x,0), & x\in \mathbb{R},\\
\underline{v}(x,R)=0, & x\in\mathbb{R},
\end{cases}
\end{equation}
where $\delta>0$ is small enough such that $0< \delta< m=\min_{[0,L]}f_v(x,0)$. A lengthy but straightforward calculation reveals, from the first two relations of \eqref{beta,w},  that $\omega$ and $\beta$ should verify
\begin{equation*}
\begin{cases}
D\omega^2\le -\frac{\mu d\beta \cos(\beta R)}{d\beta\cos(\beta R)+\nu \sin(\beta R)},\\
d\omega^2+d\beta^2\le m-\delta.
\end{cases}
\end{equation*}

Because of \eqref{R-condition}, $\delta>0$ can be chosen sufficiently small such that 
$d\big(\frac{\pi}{2R}\big)^2<m-\delta$. Then, $\beta$ can be  chosen  very closely to $\frac{\pi}{2R}$, say $\beta\sim\frac{\pi}{2R}$ and $\frac{\pi}{2R}<\beta<
\frac{\pi}{R}$, such that
\begin{equation*}
\kappa:=\min\bigg\{- \frac{\mu d\beta \cos(\beta R)}{D(d\beta\cos(\beta R)+\nu \sin(\beta R))}, \frac{m-\delta}{d}-\beta^2 \bigg\}>0.
\end{equation*}

Therefore, $(\underline{u},\underline{v})$ satisfies \eqref{beta,w}, provided $\omega^2\le \kappa$. On the other hand, $u(t,x)>0$ and $v(t,x,y)>0$ for all $t>0$ and $(x,y)\in \mathbb{R}\times [0,R)$, and $\partial_y v(t,x,R)<0$ for all $t>0$ and $x\in\mathbb{R}$, which is a direct consequence of  Proposition \ref{cp-strip} and the Hopf lemma. Hence, there is $\varep_0>0$ such that $\varep(\underline{u},\underline{v})\le (u(1,\cdot),v(1,\cdot,\cdot))$ in $\overline\Omega_R$ for all $\varep\in(0,\varep_0]$. It then follows from the same lines as in Theorem \ref{liouville} that there is a nontrivial steady state $(U_1,V_1)$ of \eqref{pb in strip} such that  
\begin{equation}\label{lower-1}
\varep\underline{u}\le U_1\le\liminf\limits_{t\to +\infty}u(t,x),\qquad \varep\underline{v}\le V_1\le\liminf\limits_{ t\to+\infty}v(t,x,y),
\end{equation}
locally uniformly in $\overline{\Omega}_R$, thanks to Proposition \ref{cp--generalized sub-strip}. 
	On the other hand, by choosing $(\overline U,\overline V)$ as in \eqref{super-constant sol} and by using the same argument as in Theorem \ref{liouville}, it comes that there is 
a stationary solution $(U_2, V_2)$ of \eqref{pb in strip} satisfying 
\begin{equation}\label{upper-1}
\limsup\limits_{ t\to+\infty}u(t,x)\le U_2\le\overline U,\qquad \limsup\limits_{ t\to+\infty}v(t,x,y)\le V_2\le\overline V,
\end{equation} 
locally uniformly in $\overline{\Omega}_R$.
Therefore, the existence part is proved.

Moreover, let $(U,V)$ be a nonnegative bounded stationary solution of \eqref{pb in strip}. 
From the analysis above and from the elliptic strong maximum principle, one also deduces that, for any given $\hat x\in\mathbb{R}$, for $\forall (x,y)\in(\hat x-\frac{\pi}{2\omega},\hat x+\frac{\pi}{2\omega})\times[0,R)$,
\begin{equation*}
U(x)>\varep\cos(\omega(x-\hat x))>0,~~V(x,y)>\varep\cos(\omega(x-\hat x))\frac{\mu\sin(\beta(R-y))}{d\beta\cos(\beta R)+\nu\sin(\beta R)}>0,~~\text{for all}~\varep\in(0,\varep_0],
\end{equation*}
which implies $\inf_\mathbb{R}U>0$ and  $\inf_{\mathbb{R}\times[0,R)}V>0$. Then, by repeating the uniqueness argument in the proof of Theorem \ref{liouville} and by \eqref{lower-1}--\eqref{upper-1}, the conclusion in Theorem \ref{thm2.1} follows.
\end{proof}

 In the sequel, we show the limiting behavior of the steady state $(U_R,V_R)$ of \eqref{pb in strip} as $R$ goes to infinity, which will play a crucial role in  obtaining the existence of  pulsating fronts in the half-plane $\Omega$ in Section \ref{section5}. 

\begin{proposition}
	\label{prop3.9}
The stationary solution
 $(U_R,V_R)$ of \eqref{pb in strip} satisfies the following properties:
 \begin{itemize}
 	\item[(i)]   $0<U_R<{\nu}/{\mu}$ in $\mathbb{R}$, $0< V_R<1$ in $\mathbb{R}\times [0,R)$;
 	\item[(ii)] the limiting property holds:
		\begin{equation}
		\label{truncated to half-plane}
		(U_R(x),V_R(x,y))\to ({\nu}/{\mu}, 1)\quad \text{as}\ R\to+\infty
		\end{equation}
		 uniformly in $x\in\mathbb{R}$ and locally uniformly in $y\in[0,+\infty)$.
		  \end{itemize}
\end{proposition}
\begin{proof} 
	 
	 (i) From the proof of Theorem \ref{thm2.1}, it is seen that $U_R>0$ in $\mathbb{R}$ and $V_R>0$ in $\mathbb{R}\times[0,R)$. Notice also that $(\nu/\mu,1)$ is obviously a strict stationary supersolution of \eqref{pb in strip}. Let $(\overline u, \overline v)$ be the unique bounded classical solution of \eqref{pb in strip} with initial condition $(\nu/\mu,1)$. It follows from Proposition \ref{cp-strip} that $(\overline u,\overline v)$ is decreasing in $t$. Since $(\overline u(t,\cdot),\overline v(t,\cdot,\cdot))$ converges to $(U_R,V_R)$ as $t\to+\infty$ locally uniformly in $\overline\Omega_R$ by Theorem \ref{thm2.1}, one has $\overline u(t,x)>U_R(x)$ and $\overline v(t,x,y)> V_R(x,y)$ for all $t\ge 0$ and $(x,y)\in\mathbb{R}\times[0,+\infty)$. Therefore, $U_R<\nu/\mu$ in $\mathbb{R}$ and $V_R<1$ in $\mathbb{R}\times[0,+\infty)$. The statement (i) is then proved.
	 
	
	(ii) Now, let us turn to show the limiting property. First, we claim that $(U_R,V_R)$ is increasing in $R$. To prove this, we fix $R_1<R_2$. Denote by $\Omega_i:=\Omega_{R_i}$ ($i=1,2$) and by $(U_i, V_i):=(U_{R_i},V_{R_i})$($i=1,2$) the unique nontrivial stationary solutions of \eqref{pb in strip} in $\overline\Omega_i$ $(i=1,2)$, respectively. One can prove that
	$U_1<U_2$ in $\mathbb{R}$ and $V_1<V_2$ in $\mathbb{R}\times[0,R_1)$, by noticing that $(U_2,V_2)$ is a strict stationary supersolution of \eqref{pb in strip} in $\Omega_1$ and by a similar argument as in (i). Our claim is thereby proved. Due to the boundedness of $(U_R,V_R)$
	in (i), it follows from the monotone convergence theorem and  standard elliptic estimates that $(U_R,V_R)$ converges as $R\to+\infty$ locally uniformly in $\overline\Omega$ to  a classical solution $(U, V)$ of the following stationary problem:
	\begin{equation*}
	\begin{cases}
	-D\partial_{xx}U=\nu V(x,0)-\mu U, & x\in \mathbb{R},  \\
	-d\Delta V =f(x,V),  &(x,y)\in\Omega,\\
	-d\partial_y V(x,0)=\mu U(x)-\nu V(x,0), & x\in \mathbb{R}.
	\end{cases}
	\end{equation*} 
	Owing to Theorem \ref{liouville}, it follows that $(U, V)=({\nu}/{\mu}, 1)$. Thus, \eqref{truncated to half-plane} is proved.
	\end{proof}

 \section{Propagation properties in the strip: Proofs of Theorems \ref{thm-asp-strip} and  \ref{thm-PTF-in strip}}
 \label{section4}
 
 	This section is dedicated to  the existence of the asymptotic spreading speed $c^*_R$ and its coincidence with the minimal wave speed for pulsating fronts  for truncated problem \eqref{pb in strip} along the road. In particular, we will give  a variational formula for  $c^*_R$ by using the principal eigenvalue for certain linear eigenvalue problem. As will be shown below, the discussion combines the dynamical system approach for monostable evolution problems developed in \cite{LZ2010} with PDE's method.
 

  Let $D:=[0,L]\times[0,R]$ and define the Banach spaces
 	$$X=\{(u,v)\in C([0,L])\times C(D): v(\cdot,R)=0~\text{in}~[0,L]\}$$ with the norm $\Vert(u,v)\Vert_X=\Vert u\Vert_{C([0,L])}+\Vert v\Vert_{C(D)}$, then $(X,X^+)$ is an ordered Banach space and the cone $X^+$ has empty interior. Let $Y$ be a closed subspace of $X$ given by
 	$$Y=\{(u,v)\in C^1([0,L])\times C^1(D): v(\cdot,R)=0~\text{in}~[0,L]\}$$ with the norm $\Vert (u,v)\Vert_Y =\Vert u\Vert_{C^1([0,L])}+\Vert v\Vert_{C^1(D)}$. It is seen that the inclusion map $Y\hookrightarrow X$ is a continuous linear map. Moreover, the cone $Y^+$  has nonempty interior $\text{Int}(Y^+)$, see e.g. \cite[Corollary 4.2]{Smith1995}, given by
 	\begin{align*}
 	\text{Int}(Y^+)=\{(u,v)\in Y^+: (u,v)>(0,0)~ \text{in}~ [0,L]\times[0,R),~ \partial_y v(\cdot,R)<0~\text{in}~[0,L]\}.
 	\end{align*}
 	We write $(u_1,v_1)\ll (u_2,v_2)$ if $(u_1,v_1),(u_2,v_2)\in Y$ and $(u_2,v_2)-(u_1,v_1)\in\text{Int}(Y^+)$.

   Set $\mathcal{H}:=L\mathbb{Z}$. 
   We use $\mathcal{C}$ to denote the set of all bounded and continuous function pairs from $\mathcal{H}$ to $X$, and $\mathcal{D}$ to denote the set of all bounded and continuous function pairs from $\mathcal{H}$ to $Y$. Moreover, any element in $X$ ($Y$) can be regarded as a constant function in $\mathcal{C}$ ($\mathcal{D}$).

   For any $u,v\in\mathcal{C}$, we write $u\ge v$ provided $u(x)\ge v(x)$  for all $x\in\mathcal{H}$, $u>v$  provided $u\ge v$ but $u\neq v$.  
   For $u,v\in\mathcal{D}$, we write $u\gg v$ provided $u(x)\gg v(x)$ for all $x\in\mathcal{H}$. We equip $\mathcal{C}$ $(\mathcal{D})$     with the compact open topology, i.e., $(u_n,v_n)\to (u,v)$ in $\mathcal{C}$ ($\mathcal{D}$) means that  $u_n(x)\to u(x)$ uniformly for $x$ in every compact interval of $\mathbb{R}$ and $v_n(x,y)\to v(x,y)$ uniformly for $(x,y)$
   in every compact subset of $\overline\Omega_R$.

 Define 
  \begin{align*}
  &\mathbb{C}_0:=\{(u,v)\in C(\mathbb{R})\times C(\overline\Omega_R):~ v(\cdot,R)=0~ \text{in}~\mathbb{R}\},\cr
  &\mathbb{C}^1_0:=\{(u,v)\in C^1(\mathbb{R})\times C^1(\overline\Omega_R): v(\cdot,R)=0~\text{in}~\mathbb{R}\}.
  \end{align*}    
   Any continuous and bounded function pair $(u,v)$ in $\mathbb{C}_0$ can be regarded as a function pair $(u(z),v(z))$ in the space $\mathcal{C}$ in the sense that $\big(u(z)(x),v(z)(x,y)\big):=\big(u(x+z),v(x+z,y)\big)$ for all $z\in\mathcal{H}$ and $(x,y)\in D$. In this sense, $(U_R,V_R)\in\mathcal{C}$ and the set
   $$\mathcal{K}:=\big\{(u,v)\in C(\mathbb{R})\times C(\overline\Omega_R): (0,0)\le (u,v)\le (U_R,V_R)~\text{in}~\overline\Omega_R\big\}$$
   is a closed subset of $\mathcal{C}_{(U_R,V_R)}:=\{(u,v)\in\mathcal{C}: (0,0)\le (u,v)\le (U_R,V_R)\}$ and satisfies (K1)--(K5) in \cite{LZ2010}.
   
Define a family of operators $\{Q_t\}_{t\ge 0}$ on $\mathcal{K}$ by
	 \begin{equation*}
	 Q_t[(u_0,v_0)]:=(u(t,\cdot;u_0), v(t,\cdot,\cdot;v_0))~~\text{for}~ (u_0,v_0)\in \mathcal{K},
	 \end{equation*}  
	 where $(u(t,\cdot;u_0), v(t,\cdot,\cdot;v_0))$ is the  solution of system \eqref{pb in strip} with initial datum $(u_0,v_0)\in \mathcal{K}$.
	  It is easily seen that $Q_0[(u_0,v_0)]=(u_0,v_0)$ for all $(u_0,v_0)\in\mathcal{K}$, and $Q_{t}\circ Q_{s}[(u_0,v_0)]=Q_{t+s}[(u_0,v_0)]$ for any $t,s\ge 0$ and $(u_0,v_0)\in\mathcal{K}$. For any given $(u_0,v_0)\in\mathcal{K}$, it can be deduced from Proposition \ref{wellposedness-strip} that $Q_t[(u_0,v_0)]$ is continuous in $t\in[0,+\infty)$ with respect to the compact open topology.

     Assume that $(u_k,v_k)$ and $(u,v)$ are the unique solutions to \eqref{pb in strip} with initial data $(u_{0k},v_{0k})$ and $(u_0,v_0)$ in $\mathcal{K}$, respectively. Suppose that $(u_{0k},v_{0k})\to(u_0,v_0)$ as $k\to+\infty$ locally uniformly in $\overline\Omega_R$, we claim that $(u_{k},v_{k})\to(u,v)$ as $k\to+\infty$ in $C^{1,2}_{loc}([0,+\infty)\times\overline\Omega_R)$, which will imply that $Q_t[(u_0,v_0)]$ is continuous in $(u_0,v_0)\in\mathcal{K}$ uniformly in $t\in[0,T]$ for any $T>0$.
     To prove this, we define a smooth cut-off function $\chi^n: \mathbb{R}\mapsto[0,1]$ such that $\chi^n(\cdot)=1$ in $[-n+1,n-1]$ and $\chi^n(\cdot)=0$ in $\mathbb{R}\backslash[-n,n]$. Then, $(\chi^n u_{0k},\chi^n v_{0k})\to (\chi^n u_0,\chi^n v_0)$ as $k\to+\infty$ uniformly in $[-n,n]$. Let $(u^n_k,v^n_k)$ and $(u^n,v^n)$ be the solutions to  \eqref{pb in strip-cut off} in $D_1:=[-n,n]\times[0,R]$ with initial data
     $(\chi^n u_{0k},\chi^n v_{0k})$ and $(\chi^n u_0,\chi^n v_0)$, respectively.
     One can choose two positive bounded and monotone  function sequences $(\underline u^n_{0k},\underline v^n_{0k})$ and $(\bar u^n_{0k},\bar v^n_{0k})$ in the space $\left\{(u,v)\in C^\infty([-n,n])\times C^\infty(D_1): u(\pm n)=0, v(\cdot,R)=0~ \text{in}~ [-n,n], v(\pm n,\cdot)=0~\text{in}~[0,R]\right\}$, such that
     \begin{align*}
(0,0)\le (\underline u^n_{0k},\underline v^n_{0k})&\le (\chi^n u_{0k},\chi^n v_{0k})\le  (\bar u^n_{0k},\bar v^n_{0k}),\\
(\underline u^n_{0k},\underline v^n_{0k})\nearrow (\chi^n u_0,\chi^n v_0),&~~(\bar u^n_{0k},\bar v^n_{0k})\searrow (\chi^n u_0,\chi^n v_0)~~\text{uniformly in}~D_1~\text{as}~k\to+\infty.
     \end{align*}
 By a comparison argument, it follows that
 \begin{align*}
 (\underline u^n_k,\underline v^n_k)\le (\underline u^n_{k+1},\underline v^n_{k+1})\le (u^n_{k+1},v^n_{k+1})\le (\bar u^n_{k+1},\bar v^n_{k+1})\le (\bar u^n_k,\bar v^n_k)~\text{for all}~t>0~\text{and}~(x,y)\in D_1,
 \end{align*}    
 where $(\underline u^n_k,\underline v^n_k)$ and $(\bar u^n_k,\bar v^n_k)$ are the classical solutions to \eqref{pb in strip-cut off} with initial data $(\underline u^n_{0k},\underline v^n_{0k})$ and $(\bar u^n_{0k},\bar v^n_{0k})$, respectively. From standard parabolic estimates, the functions $(\underline u^n_k,\underline v^n_k)$ and $(\bar u^n_k,\bar v^n_k)$ converge  to $(\underline u^n,\underline v^n)$ and $(\bar u^n,\bar v^n)$ as $k\to+\infty$ in $C^{1+\alpha/2,2+\alpha}_{loc}([0,+\infty)\times D_1)$, respectively. Moreover, $(\underline u^n,\underline v^n)$ and $(\bar u^n,\bar v^n)$ are classical solutions to \eqref{pb in strip-cut off}. Since 
 \begin{align*}
 \lim_{
 	t\to 0,k\to +\infty
 	}(\underline u^n_k (t,\cdot),\underline v^n_k(t,\cdot,\cdot)) = \lim_{
 	t\to 0,k\to +\infty
 } (\bar u^n_k (t,\cdot),\bar v^n_k (t,\cdot,\cdot))=(\chi^n u_0,\chi^n v_0),
 \end{align*}
 uniformly in $(x,y)\in D_1$, therefore
 \begin{align*}
  \lim_{
 		t\to 0,k\to +\infty
 }(u^n_{k}(t,\cdot),v^n_{k}(t,\cdot,\cdot))=(\chi^n u_0,\chi^n v_0)= \lim_{
 t\to 0} (u^n(t,\cdot),v^n(t,\cdot,\cdot)),
 \end{align*}
 uniformly in $(x,y)\in D_1$, and by the uniqueness of the solutions to \eqref{pb in strip-cut off}, it follows that
 $$(\underline u^n,\underline v^n)=(\bar u^n,\bar v^n)=(u^n,v^n)~\text{for}~t>0~\text{and}~(x,y)\in D_1.$$
 Hence, $(u^n_k,v^n_k)\to (u^n,v^n)$ as $k\to+\infty$
 in $C^{1+\alpha/2,2+\alpha}([0,T]\times D_1)$ for any $T>0$. By the approximation argument and parabolic estimates, $(u^n_k,v^n_k)$ and $(u^n,v^n)$ converge, respectively, to $(u_k,v_k)$ and $(u,v)$ as $n\to+\infty$ (at least) in $C^{1,2}_{loc}([0,+\infty)\times\overline\Omega_R)$. Consequently, $(u_k,v_k)\to (u,v)$ as $k\to+\infty$ in $C^{1,2}_{loc}([0,+\infty)\times\overline\Omega_R)$.

	  From the observation that for any $t,s\ge 0$ and for  $(u_0,v_0), (\tilde u_0,\tilde v_0)\in\mathcal{K}$,
	 \begin{equation*}
	 \big|Q_t[(u_0,v_0)]-Q_s[(\tilde u_0,\tilde v_0)]\big|\le \big|Q_t[(u_0,v_0)]-Q_t[(\tilde u_0,\tilde v_0)]\big|+\big|Q_t[(\tilde u_0,\tilde v_0)]-Q_s[(\tilde u_0,\tilde v_0)]\big|,
	 \end{equation*}
	 it comes that $Q_t[(u_0,v_0)]$ is continuous in
	   $(t,(u_0,v_0))\in[0,T]\times \mathcal{K}$. Note that for any $t>0$, it can be expressed as $t=mT+t'$ for some $m\in\mathbb{Z}_+$ and $t'\in[0,T)$. Hence, $Q_t[(u_0,v_0)]=(Q_T)^m Q_{t'}[(u_0,v_0)]$. Thus,  $Q_t[(u_0,v_0)]$ is continuous in $(t,(u_0,v_0))\in[0,+\infty)\times \mathcal{K}$. Therefore, it follows that $\{Q_t\}_{t\ge 0}$ is a continuous-time semiflow.
	 We claim that $\{Q_t\}_{t\ge 0}$ is subhomogeneous on $\mathcal{K}$ in the sense that $Q_t[\kappa(u_0,v_0)]\ge \kappa Q_t[(u_0,v_0)]$ for all $\kappa\in[0,1]$ and for all $(u_0,v_0)\in\mathcal{K}$. The case that $\kappa=0,1$ is trivial. Suppose now that $\kappa\in(0,1)$. Define	 
	 \begin{equation*}
	 (\overline u,\overline v)=(u(t,\cdot;\kappa u_0), v(t,\cdot,\cdot;\kappa v_0)),~~~(\underline u,\underline v)=\kappa (u(t,\cdot;u_0),v(t,\cdot,\cdot;v_0)).
	 \end{equation*}
	 From Proposition \ref{cp-strip}, it follows that $(\overline u,\overline v)$ and $(\underline u,\underline v)$ belong to $\mathcal{K}$. Moreover, $(\overline u,\overline v)$ and $(\underline u,\underline v)$ satisfy, respectively,
	 	 \begin{equation*}
	 \begin{cases}
	 \partial_t\overline u-D\partial_{xx}\overline{u}= \nu\overline{v}(t,x,0)-\mu \overline{u}, &t>0, x\in \mathbb{R}, \\
	 \partial_t \overline v-d\Delta\overline{v} = f(x,\overline{v}),  &t>0, (x,y)\in\Omega_R,\\
	 -d\partial_y \overline{v}(t,x,0)=\mu \overline{u}-\nu \overline{v}(t,x,0), &t>0,  x\in \mathbb{R},\\
	 \overline{v}(t,x,R)=0, &t>0, x\in\mathbb{R},\\
	 (\overline u_0,\overline v_0)=\kappa(u_0,v_0),
	 \end{cases}
	 \end{equation*}
	 and
	 \begin{equation*}
	 \begin{cases}
	 \partial_t\underline u-D\partial_{xx}\underline{u}= \nu\underline{v}(t,x,0)-\mu \underline{u}, &t>0, x\in \mathbb{R}, \\
	 \partial_t \underline v-d\Delta\underline{v} < f(x,\underline{v}),  &t>0, (x,y)\in\Omega_R,\\
	 -d\partial_y \underline{v}(t,x,0)=\mu \underline{u}-\nu \underline{v}(t,x,0), &t>0,  x\in \mathbb{R},\\
	 \underline{v}(t,x,R)=0, &t>0, x\in\mathbb{R},\\
	  (\underline u_0,\underline v_0)=\kappa(u_0,v_0),
	 \end{cases}
	 \end{equation*}
	 by using the assumption that $f(x,v)/v$ is decreasing in $v>0$ for all $x\in\mathbb{R}$.
	 Proposition \ref{cp-strip} then yields that $\overline u(t,x)\ge \underline u(t,x)$ and  $\overline v(t,x,y)\ge \underline v(t,x,y)$ for all $t\ge 0$ and $(x,y)\in\overline\Omega_R$. This proves our claim.
	 By classical parabolic theory, together with Propositions \ref{cp-strip}--\ref{wellposedness-strip} and Theorem \ref{thm2.1}, for each $t>0$, the solution map $Q_t:\mathcal{K}\to\mathcal{K}$ satisfies the following properties:
	   \begin{itemize}
	 	\item[(A1)]  $Q_t[T_a[(u_0,v_0)]]=T_a[Q_t[(u_0,v_0)]]$ for all $(u_0,v_0)\in\mathcal{K}$ and $a\in \mathcal{H}$, where $T_a$  is a shift operator defined by $T_a[(u(t,x),v(t,x,y))]=(u(t,x-a),v(t,x-a,y))$.
	 	
	 	\item[(A2)]   $Q_t[\mathcal{K}]$ is uniformly bounded and $Q_t:\mathcal{K}\to \mathcal{D}$ is  continuous with respect to the compact open topology, due to the analysis above. 
	 	 \item[(A3)]  $Q_t:\mathcal{K}\to \mathcal{D}$ is compact with respect to the compact open topology, which follows from Proposition \ref{wellposedness-strip}.
	 	 
	 	 \item[(A4)]  $Q_t: \mathcal{K}\to \mathcal{K}$ is  monotone (order-preserving) in the sense that if $(u_{01},v_{01})$ and $(u_{02},v_{02})$ belong to $ \mathcal{K}$ satisfying $u_{01}\le u_{02}$ in $\mathbb{R}$ and  $v_{01}\le v_{02}$ in $\overline{\Omega}_R$, then $u(t,x;u_{01})\le u(t,x;u_{02})$ and $ v(t,x,y;v_{01})\le v(t,x,y;v_{02})$ for all $t>0$ and  $(x,y)\in\overline{\Omega}_R$. This follows from  Proposition \ref{cp-strip}.
	 	 
\item[(A5)] 
 $Q_t$ admits exactly two fixed points $(0,0)$ and $(U_R,V_R)$ in $Y$. Let $(u(t,x;u_0),v(t,x,y;v_0))$ be the solution of \eqref{pb in strip} with $L$-periodic (in $x$)  initial value  $(u_0,v_0)\in\mathcal{K}\cap Y$  satisfying $(0,0)\ll(u_0,v_0)\le (U_R,V_R)$, it comes that
	 	 \begin{equation}
	 	 \label{long time -periodic initial}
	 	 \lim_{t\to +\infty} (u(t,x;u_0),v(t,x,y;v_0))= (U_R(x),V_R(x,y))~~\text{uniformly in}~(x,y)\in\overline\Omega_R.
	 	 \end{equation}
	  Indeed,  Theorem \ref{thm2.1} implies that  $(U_R,V_R)$
	 	 is the unique $L$-periodic positive steady state     of \eqref{pb in strip}.  Moreover, \eqref{long time -periodic initial} can be achieved by a similar argument to that of Theorem \ref{thm2.1}.
  \end{itemize}

Therefore, $Q_t$ is a subhomogeneous semiflow on $\mathcal{K}$ and satisfies hypotheses (A1)--(A5) in \cite{LZ2010} for any $t>0$. Moreover, it is straightforward to check that assumption (A6) in \cite{LZ2010} is also satisfied.
  In particular, $Q_1$ satisfies (A1)--(A6) in \cite{LZ2010}. By Theorem 3.1 and Proposition 3.2 in \cite{LZ2010}, it then follows that the solution map $Q_1$ admits rightward and leftward spreading speeds $c^*_{R,\pm}$. Furthermore, Theorems 4.1--4.2 in \cite{LZ2010} imply that  $Q_1$ has a rightward periodic traveling wave $(\phi_R(x-c,x),\psi_R(x-c,x,y))$ connecting $(U_R,V_R)$ and $(0,0)$ such that $(\phi_R(s,x),\psi_R(s,x,y))$ is non-increasing in $s$ if and only if $c\ge c^*_{R,+}$. Similar results holds for  leftward periodic traveling waves with minimal wave speed $c^*_{R,-}$.



 To obtain the variational formulas for $c^*_{R,\pm}$, we use the linear operators approach. 
   Let us consider the linearization  of the truncated problem  \eqref{pb in strip} at its zero solution:
 \begin{equation}
 \label{linear pb in strip}
 \begin{cases}
 \partial_t u -D\partial_{xx}u=\nu v(t,x,0)-\mu u, &t>0,\  x\in \mathbb{R}, \\
 \partial_t v -d\Delta v =f_v(x,0)v,  &t>0,\ (x,y)\in\Omega_R,\\
 -d\partial_y v(t,x,0)=\mu u-\nu v(t,x,0), &t>0,\  x\in \mathbb{R},\\
 v(t,x,R)=0, &t>0,\  x\in \mathbb{R}.
 \end{cases}
 \end{equation}
 Let $\{L(t)\}_{t\ge 0}$ be the linear solution semigroup generated by \eqref{linear pb in strip}, that is, $L(t)[(u_0, v_0)]=(u_t(u_0),v_t(v_0))$, where $(u_t(u_0),v_t(v_0)):=(u(t,\cdot;u_0),v(t,\cdot,\cdot;v_0))$ is the solution of \eqref{linear pb in strip} with initial value $(u_0, v_0)\in \mathcal{D}$. For any given $\alpha\in\mathbb{R}$, substituting $(u(t,x),v(t,x,y))=e^{-\alpha x}(\widetilde{u}(t,x),\widetilde{v}(t,x,y))$ in \eqref{linear pb in strip} yields 
 \begin{equation}
 \label{modified linear pb in strip}
 \begin{cases}
 \partial_t \widetilde u -D\partial_{xx}\widetilde u+2D\alpha\partial_x \widetilde{u}-D\alpha^2\widetilde u=\nu\widetilde v(t,x,0)-\mu\widetilde u, &t>0,\  x\in \mathbb{R}, \\
 \partial_t\widetilde v -d\Delta\widetilde v +2d\alpha\partial_x \widetilde{v}-d\alpha^2\widetilde{v} =f_v(x,0)\widetilde v,  &t>0,\ (x,y)\in\Omega_R,\\
 -d\partial_y\widetilde v(t,x,0)=\mu \widetilde u-\nu\widetilde v(t,x,0), &t>0,\  x\in \mathbb{R},\\
 \widetilde v(t,x,R)=0, &t>0,\  x\in \mathbb{R}. 
 \end{cases}
 \end{equation}
 Let $\{L_\alpha (t)\}_{t\ge 0}$ be the linear solution semigroup generated by \eqref{modified linear pb in strip}, then one has $L_\alpha(t)[(\widetilde u_0, \widetilde v_0)]=(\widetilde{u}_t(\widetilde u_0),\widetilde{v}_t(\widetilde v_0))$, where $(\widetilde{u}_t(\widetilde u_0),\widetilde{v}_t(\widetilde v_0)):=(\widetilde u(t,\cdot;\widetilde u_0),\widetilde v(t,\cdot,\cdot;\widetilde v_0))$ is the solution of \eqref{modified linear pb in strip} with initial value $(\widetilde u_0, \widetilde v_0)=(u_0,v_0)e^{\alpha x}$. It then follows that, for any $(\widetilde u_0, \widetilde v_0)\in\mathcal{D}$,
 \begin{equation*}
 L(t)[e^{-\alpha x}(\widetilde u_0, \widetilde v_0)]=e^{-\alpha x}L_\alpha (t)[(\widetilde u_0, \widetilde v_0)]~~~~\text{for}~ t\ge 0~\text{and}~ (x,y)\in\overline{\Omega}_R.
 \end{equation*}
Substituting $(\tilde u(t,x), \tilde v(t,x,y))=e^{-\sigma t}(p(x),q(x,y))$, with $p,q$ periodic (in $x$), into \eqref{modified linear pb in strip} leads to the following periodic eigenvalue problem:	\begin{align}
\label{5.15/5.18'}
\begin{cases}
\mathcal{L}_{1,\alpha}(p, q):=-Dp''+2D\alpha p'+(-D\alpha^2+\mu) p-\nu q(x,0)=\sigma p,  &x\in\mathbb{R},\\
\mathcal{L}_{2,\alpha}( p, q):=-d\Delta q+2d\alpha\partial_x q-(d\alpha^2+f_v(x,0)) q=\sigma q,  &(x,y)\in\Omega_R,\\
\mathcal{B}( p, q):=-d\partial_y q(x,0)+\nu q(x,0)-\mu p=0, \   &x\in\mathbb{R},\\
q(x,R)=0, \   &x\in\mathbb{R},\\
p,  q \ \text{are}~ L\text{-periodic with respect to} \ x.
\end{cases}
\end{align}

 Recall that  $M:=\max_{[0,L]}f_v(x,0)$ and $m:=\min_{[0,L]}f_v(x,0)$.
 We have:  
 \begin{proposition}
 	\label{principal eigenvalue in strip} 
Set $\zeta(x):=f_v(x,0)$.
For all $\alpha\in\mathbb{R}$, the periodic eigenvalue problem \eqref{5.15/5.18'} admits the principal eigenvalue $\lambda_{R,\zeta}(\alpha)\in\mathbb{R}$ with a unique (up to multiplication by some  constant) positive periodic (in $x$) eigenfunction pair $(p,q)$ belonging to  $ C^{3}(\mathbb{R})\times C^{3}(\overline\Omega_R)$.	Moreover, $\lambda_{R,\zeta}(\alpha)$ has the following properties:
	\begin{enumerate}[(i)]
	\item For all $\alpha\in\mathbb{R}$, the principal eigenvalue $\lambda_{R,\zeta}(\alpha)$ is equal to 
	\begin{equation}
	\label{5.16'}
	\lambda_{R,\zeta}(\alpha)=\max_{( p, q)\in \varSigma}\min\bigg\{\inf_{\mathbb{R}}\frac{\mathcal{L}_{1,\alpha}( p, q)}{ p}, \inf_{\mathbb{R}\times[0,R)}\frac{\mathcal{L}_{2,\alpha}( p, q)}{ q}\bigg\},
	\end{equation}
	where 
	\begin{align*}
	\varSigma:=&\big\{( p, q)\in C^2(\mathbb{R})\times C^{2}(\overline\Omega_R):~ p>0~\text{in}~\mathbb{R}, q>0~\text{in} \ \mathbb{R}\times[0,R),~~~~~~~~~~~~~\\ &~p,q~\text{are}~L\text{-periodic in}~x,~\mathcal{B}(p, q)= 0 \ \text{in} \ \mathbb{R},
	~ \partial_y q(\cdot,R)<0= q(\cdot,R)~\text{in}~\mathbb{R}\big\}.
	\end{align*} 
	
	\item For fixed $R$ and for all $\alpha\in\mathbb{R}$, $\zeta\mapsto\lambda_{R,\zeta}(\alpha)$ is  non-increasing in the sense that, if $\zeta_1(x)\le \zeta_2(x)$ for all $x\in\mathbb{R}$, then $\lambda_{R,\zeta_1}(\alpha)\ge\lambda_{R,\zeta_2}(\alpha)$. Moreover, $\lambda_{R,\zeta}(\alpha)$ is continuous  with respect to $\zeta$ in the sense that, if $\zeta_n\to  \zeta$, then $\lambda_{R,\zeta_n}(\alpha)\to\lambda_{R,\zeta}(\alpha)$. 

	\item For all $\alpha\in\mathbb{R}$, $ \lambda_{R,\zeta}(\alpha)$ is decreasing with respect to $R$.

	\item For fixed $R$, $\alpha\mapsto\lambda_{R,\zeta}(\alpha)$ is  concave in $\mathbb{R}$ and satisfies
	\begin{align}
	\label{bound'}
\max\Big\{D\alpha^2-\mu,\  d\alpha^2+m-d\frac{\pi^2}{R^2}\Big\}<-\lambda_{R,\zeta}(\alpha)<\max\Big\{D\alpha^2+\nu-\mu+\frac{\mu\nu}{d},\  d(\alpha^2+1)+M\Big\}.
	\end{align}
	\end{enumerate}
\end{proposition}
\begin{proof}[Proof of Proposition \ref{principal eigenvalue in strip}]
	The proof is divided into six  steps.
	
	\textit{Step 1. Solving the eigenvalue problem \eqref{5.15/5.18'}.} Set $\Lambda_\zeta(\alpha):=\max\big\{D\alpha^2+\nu-\mu+{\mu\nu}/{d}, d(\alpha^2+1)+M\big\}$.  We introduce a Banach space $\mathcal{F}$ of periodic (in $x$) function pairs $(u,v)$ belonging to $C^{1}(\mathbb{R})\times C^{1}(\overline\Omega_R)$ such that $v(\cdot,R)=0$ in $\mathbb{R}$,
	equipped with $\Vert(u,v)\Vert_{\mathcal{F}}=\Vert u\Vert_{C^1([0,L])}+\Vert u\Vert_{C^1([0,L]\times[0,R])}$. For any $(g_1, g_2)\in \mathcal{F}$  and $\Lambda\ge \Lambda_{\zeta}(\alpha)$, let us consider the modified problem:
	\begin{align}
	\label{5.17'}
	\begin{cases}
	\mathcal{L}_{1,\alpha}( p, q)+\Lambda p=g_1, \  &x\in\mathbb{R},\\
	\mathcal{L}_{2,\alpha}( p, q)+\Lambda q=g_2, \  &(x,y)\in\Omega_R,\\
	\mathcal{B}( p, q)=0, \  \quad &x\in\mathbb{R},\\
	q(x,R)=0, \  \quad &x\in\mathbb{R},\\
	p,  q \ \text{are} \ L\text{-periodic with respect to} \ x.
	\end{cases}
	\end{align}
	First, we construct  ordered super- and subsolutions for problem \eqref{5.17'}. Set $(\overline p,\overline q)=K(1,1+\frac{\mu}{d}e^{-y})$. Choosing $K>0$ large enough (depending only on $\Vert g_1\Vert_{L^\infty(\mathbb{R})}$ and $\Vert g_2\Vert_{L^\infty(\overline\Omega_R)}$ if $g_1$, $g_2$ are positive), it follows that  $(\overline p,\overline q)$ is indeed a strict supersolution of \eqref{5.17'}. By linearity of \eqref{5.17'}, up to increasing $K$ (depending only on $\Vert g_1\Vert_{L^\infty(\mathbb{R})}$ and $\Vert g_2\Vert_{L^\infty(\overline\Omega_R)}$ if $g_1$, $g_2$ are negative), $(\underline{ p},\underline{ q}):=-(\overline{ p}, \overline{ q})$ is a negative strict subsolution of \eqref{5.17'}. By monotone iteration method, it is known that the associated evolution problem of \eqref{5.17'}  with initial datum $(\overline p,\overline q)$ is uniquely solvable and its solution $(\overline u,\overline v)$ is  decreasing in time and is bounded from below by $(\underline p,\underline q)$  and from above by  $(\overline p,\overline q)$, respectively. From the monotone convergence theorem as well as elliptic regularity theory up to the boundary, it follows that $(\overline u,\overline v)$ converges as $t\to+\infty$ locally uniformly in $\overline\Omega_R$ to a classical periodic (in $x$) solution $(p,q)\in C^{3}(\mathbb{R})\times C^{3}(\overline\Omega_R)$ of problem \eqref{5.17'}. To prove  uniqueness of the solution to \eqref{5.17'}, we first claim that $g_1\ge 0$ in $\mathbb{R}$, $g_2\ge 0$ in $\overline\Omega_R$ implies that $ p\ge 0$ in $\mathbb{R}$, $q\ge 0$ in $\mathbb{R}\times[0,R)$. Indeed, for any fixed nonnegative  function pair $(g_1,g_2)\in \mathcal{F}$, let $( p,  q)$ be the unique solution to \eqref{5.17'}. One can easily check that, for any $K>0$, $(\underline{ p},\underline{ q})$ defined as above is a strict subsolution of \eqref{5.17'}. Assume that $p$ or $ q$ attains a negative value somewhere in their respective domains. Define
	\begin{equation*}
	\theta^*:=\min\big\{\theta>0:~ ( p, q)\ge\theta(\underline{ p},\underline{ q}) \ \text{in} \ \overline\Omega_R\big\}.
	\end{equation*}
	Then, $\theta^*\in(0,+\infty)$. The function pair $( p-\theta^*\underline{ p}, q-\theta^*\underline{ q})$ is nonnegative
	 and at least one component attains zero somewhere in $\mathbb{R}\times[0,R)$ by noticing $(q-\theta^*\underline q)(\cdot,R)>0$ in $\mathbb{R}$. Set  $(w,z):=( p-\theta^*\underline{ p}, q-\theta^*\underline{ q})$, then it satisfies
	 \begin{align}
	 \label{4.8}
	 \begin{cases}
	 -Dw''+2D\alpha w'+(\Lambda-D\alpha^2+\mu) w-\nu z(x,0)\ge 0,  &x\in\mathbb{R},\\
	 -d\Delta z+2d\alpha\partial_x z+(\Lambda-d\alpha^2-\zeta(x)) z> 0,  &(x,y)\in\Omega_R,\\
	 -d\partial_y z(x,0)+\nu z(x,0)-\mu w > 0, \   &x\in\mathbb{R},\\
	 z(x,R)>0, \   &x\in\mathbb{R},\\
	 w,  z \ \text{are}~ L\text{-periodic with respect to} \ x.
	 \end{cases}
	 \end{align}
	 Assume first that there is $(x_0,y_0)\in\mathbb{R}\times[0,R)$ such that $z(x_0,y_0)=0$. There are two subcases. Suppose that $(x_0,y_0)\in\Omega_R$, then the strong maximum principle implies that $z\equiv 0$ in $\Omega_R$. This contradicts the strict inequality of $z$ in \eqref{4.8}, whence $z>0$ in $\Omega_R$. Suppose now that $y_0=0$ and $z(x_0,0)=0$, it follows that $\partial_y z(x_0,0)>0$. One then deduces from $-d\partial_y z(x_0,0)+\nu z(x_0,0)-\mu w(x_0)>0$ that $w(x_0)<-(d/\mu)\partial_y z(x_0,0)<0$, which is impossible since $w\ge 0$ in $\mathbb{R}$. Therefore, $z>0$ in $\overline\Omega_R$. It is seen from the first inequality of \eqref{4.8} that 
	 \begin{equation}
	 \label{4.9}
	 -Dw''+2D\alpha w'+(\Lambda-D\alpha^2+\mu) w\ge \nu z(\cdot,0)>0  ~~\text{in}~\mathbb{R}.
	 \end{equation}
	 Finally, assume that there is $x_0\in\mathbb{R}$ such that $w(x_0)=0$, then the strong maximum principle implies that $w\equiv 0$ in $\mathbb{R}$. This contradicts the strict inequality in \eqref{4.9}.
	 Consequently, $ p\ge 0$ on $\mathbb{R}$ and $ q\ge 0$ in $\overline \Omega_R$. If we further assume that $g_1\not\equiv 0$ in $\mathbb{R}$ or $g_2\not\equiv 0$ in $\mathbb{R}\times[0,R)$, then $p>0$  in $\mathbb{R}$ and $q>0$ in $\mathbb{R}\times[0,R)$. This can be proved by the strong maximum principle and by a similar argument as above.

	To prove uniqueness, we assume that $( p_1,  q_1)$ and $( p_2,  q_2)$ are two distinct solutions of \eqref{5.17'}, then $( p_1- p_2,  q_1- q_2)$ satisfies \eqref{5.17'} with $g_1=0$ and $g_2=0$. Using the result derived from above, we conclude that $ p_1\equiv p_2$ in $\mathbb{R}$, $ q_1\equiv q_2$ in $\overline\Omega_R$.

 According to \eqref{5.17'}, one defines an operator  $T:	\mathcal{F}\to \mathcal{F}$, $ (g_1,g_2)\mapsto ( p,  q)=T(g_1,g_2)$.
	Obviously, the mapping $T$ is linear. Moreover, we notice that the solution $( p,  q)$ of \eqref{5.17'} has a global bound which depends only on $\Vert g_1\Vert_{L^\infty(\mathbb{R})}$ and $\Vert g_2\Vert_{L^\infty(\overline\Omega_R)}$. By regularity estimates, $( p,  q)=T(g_1, g_2)$ belongs to $C^{3}(\mathbb{R})\times C^{3}(\overline\Omega_R)$, whence $(p,q)\in\mathcal{F}$. Therefore, $T$ is compact.
	
	Let $K$ be the cone $K=\left\{(u,v)\in\mathcal{F}:u\ge0~\text{in}~\mathbb{R}, v\ge 0~\text{in}~\overline\Omega_R\right\}$. Its interior $K^\circ=\big\{(u,v)\in \mathcal{F}:u>0~\text{in}~\mathbb{R}, v>0~\text{in}~\mathbb{R}\times[0,R)\big\}\neq \emptyset $ (for instance, 
	 $(u,v(y))=(1,1-y/R)$ belongs to $K^\circ$) and $K\cap (-K)={(0,0)}$. By the analysis above,  $T(K^\circ)\subset K^\circ$ and $T$ is strongly positive in the sense that, if $(g_1, g_2)\in K\backslash \{(0,0)\}$, then $ p>0$ in $\mathbb{R}$ and $ q>0$ in $\mathbb{R}\times[0,R)$.  
	
	From the  classical Krein-Rutman theory, there exists a unique positive real number $\lambda^*_{R,\zeta}(\alpha)$ and a unique (up to multiplication by constants) function pair $( p,  q)\in K^\circ$ such that $\lambda^*_{R,\zeta}(\alpha)T( p,  q)=( p,  q)$. 
	The  principal eigenvalue $\lambda^*_{R,\zeta}(\alpha)$ depends on $R$, $\alpha$ and $\zeta$. Set $\lambda_{R,\zeta}(\alpha):=\lambda^*_{R,\zeta}(\alpha)-\Lambda$, then the function $\lambda_{R,\zeta}(\alpha)$ takes value in $\mathbb{R}$. For each $\alpha\in\mathbb{R}$, $( p, q)$ is the unique (up to multiplication by  constants) positive eigenfunction pair of \eqref{5.15/5.18'} associated with $\lambda_{R,\zeta}(\alpha)$.
	
	\textit{Step 2. Proof of formula \eqref{5.16'}.} We notice from Step 1 that $( p,  q)\in \varSigma$. It then follows that 
	\begin{equation*}
	\lambda_{R,\zeta}(\alpha)\le\sup_{( p, q)\in \varSigma}\min\bigg\{\inf_{\mathbb{R}}\frac{\mathcal{L}_{1,\alpha}( p, q)}{ p}, \inf_{\mathbb{R}\times[0,R)}\frac{\mathcal{L}_{2,\alpha}( p, q)}{ q}\bigg\}.
	\end{equation*}
	To show the reverse inequality, assume by contradiction that there exists $( p_1, q_1)\in \varSigma$ such that
	\begin{equation*}
	\lambda_{R,\zeta}(\alpha)<\min\bigg\{\inf_{\mathbb{R}}\frac{\mathcal{L}_{1,\alpha}( p_1, q_1)}{ p_1}, \inf_{\mathbb{R}\times[0,R)}\frac{\mathcal{L}_{2,\alpha}( p_1, q_1)}{ q_1}\bigg\}.
	\end{equation*}
	Define 
	\begin{equation*}
	\theta^*:=\min\big\{\theta>0:~ \theta( p_1, q_1)\ge ( p, q) \ \text{in} \ \mathbb{R}\times[0,R)\big\}.
	\end{equation*}
	Then, $\theta^*>0$ and $(\theta^* p_1- p, \theta^* q_1- q)$ is nonnegative and two cases may occur, namely, either at least one component attains zero somewhere in $\mathbb{R}\times[0,R)$,  or $\theta^* p_1- p>0$ in $\mathbb{R}$, $\theta^* q_1-q>0$ in $[0,R)$ and $\partial_y(\theta^* q_1-q)(x_0,R)=0$ for some $x_0\in\mathbb{R}$. Set $(w,z):=(\theta^* p_1- p, \theta^* q_1- q)$, then $(w,z)$ satisfies
	\begin{align}
	\label{4.10}
	\begin{cases}
	-Dw''+2D\alpha w'+(-D\alpha^2+\mu-\lambda_{R,\zeta}(\alpha)) w-\nu z(x,0)>0,  &x\in\mathbb{R},\\
	-d\Delta z+2d\alpha\partial_x z-(d\alpha^2+\zeta(x)+\lambda_{R,\zeta}(\alpha)) z>0,  &(x,y)\in\Omega_R,\\
	-d\partial_y z(x,0)+\nu z(x,0)-\mu w=0, \   &x\in\mathbb{R},\\
	z(x,R)=0, \   &x\in\mathbb{R},\\
	w,  z \ \text{are}~ L\text{-periodic with respect to} \ x.
	\end{cases}
	\end{align}
	For the first case, assume first that there is $(x_0,y_0)\in\mathbb{R}\times[0,R)$ such that $z(x_0,y_0)=0$. We divide into two subcases. Suppose that $(x_0,y_0)\in\Omega_R$, then the strong maximum principle implies that $z\equiv 0$ in $\Omega_R$. This contradicts the strict inequality of $z$ in \eqref{4.10}, whence $z>0$ in $\Omega_R$.  Suppose now that $y_0=0$ and $z(x_0,0)=0$, it follows that $\partial_y z(x_0,0)>0$. One then deduces from $-d\partial_y z(x_0,0)+\nu z(x_0,0)-\mu w(x_0)=0$ that $w(x_0)=-(d/\mu)\partial_y z(x_0,0)<0$, which is impossible since $w\ge 0$ in $\mathbb{R}$. Therefore, $z>0$ in $\mathbb{R}\times[0,R)$. It is seen from the first inequality of \eqref{4.10} that 
	\begin{equation*}
	-Dw''+2D\alpha w'+(-D\alpha^2+\mu-\lambda_{R,\zeta}(\alpha)) w> \nu z(\cdot,0)>0  ~~\text{in}~\mathbb{R}.
	\end{equation*}
	Finally, assume that there is $x_0\in\mathbb{R}$ such that $w(x_0)=0$, then the strong maximum principle implies that $w\equiv 0$ in $\mathbb{R}$, contradicting the strict inequality above.  Consequently, one has $w>0$ in $\mathbb{R}$ and $z>0$ in $\mathbb{R}\times[0,R)$. On the other hand, by Hopf lemma it follows that $\partial_y z(\cdot,R)<0$ in $\mathbb{R}$, whence the second case is ruled out. Therefore,
		\begin{equation*}
	\lambda_{R,\zeta}(\alpha)\ge \sup_{( p, q)\in \varSigma}\min\bigg\{\inf_{\mathbb{R}}\frac{\mathcal{L}_{1,\alpha}( p, q)}{ p}, \inf_{\mathbb{R}\times[0,R)}\frac{\mathcal{L}_{2,\alpha}( p, q)}{ q}\bigg\}.
	\end{equation*}
 Therefore, formula \eqref{5.16'} is proven and the supremum is indeed  maximum since \eqref{5.16'} is reached by the function pair $(p,q)\in\Sigma_\alpha$. Therefore, (i) is proved.
		\vspace{2mm}
		
 	\textit{Step 3. Monotonicity and continuity of the function $\zeta\mapsto\lambda_{R,\zeta}(\alpha)$ for all $\alpha\in\mathbb{R}$.}	For any fixed $R$, if $\zeta_1(x)\le \zeta_2(x)$ in $\mathbb{R}$, formula \eqref{5.16'} together with the definition of the operator $\mathcal{L}_{2,\alpha}$ immediately implies that  $\lambda_{\zeta_1}(\alpha)\ge\lambda_{\zeta_2}(\alpha)$ for all $\alpha\in\mathbb{R}$.   
 	
 	Assume now that $\zeta_n\to \zeta$ as $n\to +\infty$, we have to show that $\lambda_{R,\zeta_n}(\alpha)\to \lambda_{R,\zeta}(\alpha)$ as $n\to +\infty$. Let $(\lambda_{R,\zeta_n}(\alpha) ;(p_n,q_n))$ be the principal eigenpair of \eqref{5.15/5.18'} with $\zeta$ replaced by $\zeta_n$   satisfying the normalization $\Vert p_n\Vert_{L^{\infty}(\mathbb{R})}=1$. From Step 1, it is seen that $(p_n,q_n)$ belongs to $C^{3}(\mathbb{R})\times C^{3}(\overline\Omega_R)$. By elliptic estimates, up to extraction of some subsequence, $(p_n,q_n)$ converges as $n\to +\infty$ uniformly in $\overline\Omega_R$ to a positive function pair $(p,q)\in C^{3}(\mathbb{R})\times C^{3}(\overline\Omega_R)$ which satisfies \eqref{5.15/5.18'} associated with $\tilde \lambda_{R}(\alpha)$ with normalization $\Vert p\Vert_{L^{\infty}(\mathbb{R})}=1$. By the uniqueness of the principal eigenpair of \eqref{5.15/5.18'}, it follows that $\tilde\lambda_{R,\zeta}(\alpha)=\lambda_{R,\zeta}(\alpha)$. Namely, $\lambda_{R,\zeta_n}(\alpha)\to \lambda_{R,\zeta}(\alpha)$
    as $n\to+\infty$. This completes the proof of (ii).
	\vspace{2mm}
	
	\textit{Step 4. Monotonicity of the function $R\mapsto\lambda_{R,\zeta}(\alpha)$ for all $\alpha\in\mathbb{R}$.} Fix $\alpha\in\mathbb{R}$ and choose $R_1>R_2$.  Set $\lambda_1=\lambda_{R_1,\zeta}(\alpha)$ and $\lambda_2=\lambda_{R_2,\zeta}(\alpha)$ and let $(\lambda_1;(p_1,q_1))$ and $(\lambda_2;(p_2,q_2))$ be the  eigenpairs of \eqref{5.15/5.18'} in $\overline\Omega_{R_1}$ and in $\overline\Omega_{R_2}$, respectively. Define
	\begin{equation*}
	\theta^*:=\min\{\theta>0:~ \theta(p_1,q_1)\ge(p_2,q_2) \ \text{in} \ \overline\Omega_{R_2}\}.
	\end{equation*}
	Then, $\theta^*>0$ is well-defined. The function pair $(w,z):=( \theta^*p_1-p_2, \theta^*q_1-q_2)$ is nonnegative  and at least one component attains zero somewhere in $\mathbb{R}\times[0,R_2)$ by noticing that $q_1|_{y=R_2}>q_2|_{y=R_2}=0$. Moreover, $(w,z)$ satisfies
	\begin{align}
	\label{4.11}
	\begin{cases}
	-Dw''+2D\alpha w'+(-D\alpha^2+\mu) w-\nu z(x,0)=\theta^*\lambda_1p_1-\lambda_2p_2,  &x\in\mathbb{R},\\
	-d\Delta z+2d\alpha\partial_x z-(d\alpha^2+\zeta(x)) z=\theta^*\lambda_1q_1-\lambda_2q_2,  &(x,y)\in\Omega_{R_2},\\
	-d\partial_y z(x,0)+\nu z(x,0)-\mu w=0, \   &x\in\mathbb{R},\\
	z(x,R_2)>0, \   &x\in\mathbb{R},\\
	w,  z \ \text{are}~ L\text{-periodic with respect to} \ x.
	\end{cases}
	\end{align}
	Assume that there is $x_0\in\mathbb{R}$ such that $w(x_0)=0$,  it follows from the first equation in \eqref{4.11} that 
	\begin{align*}
-Dw''(x_0)+2D\alpha w'(x_0)+(-D\alpha^2+\mu) w(x_0)-\nu z(x_0,0)=(\lambda_1-\lambda_2)p_2(x_0),
	\end{align*}
	Since the  function $w$ attains its minimum at $x_0$, one has $w'(x_0)=0$ and  $w''(x_0)\ge 0$, whence $(\lambda_1-\lambda_2)p_2(x_0)\le -\nu z(x_0,0)\le 0$,
	therefore $\lambda_1\le\lambda_2$.
	Assume now that there is $ (x_0,y_0)\in\mathbb{R}\times[0,R_2)$ such that $z(x_0, y_0)=0$, we distinguish  two subcases. Suppose that $y_0\in(0,R)$,   a similar analysis of the second equation in \eqref{4.11} as above implies
	 that $\lambda_1\le\lambda_2$.
	Otherwise, $z>0$ in $\Omega_R$ and $z(x_0,0)=0$, which leads to $w(x_0)=-(d/\mu)\partial_y z(x_0,0)<0$. This contradicts $w\ge 0$ in $\mathbb{R}$. To sum up, one obtains $\lambda_1\le\lambda_2$. Moreover, $\lambda_1=\lambda_2$ is impossible, otherwise $(p_1,q_1)$ would be a positive multiple of $(p_2,q_2)$, which contradicts $q_1|_{y=R_2}>q_2|_{y=R_2}=0$. As a consequence, $\lambda_1<\lambda_2$, namely, the function $R\mapsto\lambda_{R,\zeta}$ is decreasing. The proof of (iii) is complete.
	\vspace{2mm}
	
	\textit{Step 5. The concavity of the function $\alpha\mapsto\lambda_{R,\zeta}(\alpha)$.}	
Let $(\lambda_{R,\zeta}(\alpha); (p,q))$ be the principal eigenpair of \eqref{5.15/5.18'}. 	With the change of functions  $( p,  q)=e^{\alpha x}( \Phi,\Psi)$ in formula \eqref{5.16'}, one has
	\begin{align*}
	\frac{\mathcal{L}_{1, \alpha}( p,  q)}{ p}=\frac{-D \Phi''-\nu\Psi(x,0)}{\Phi}+\mu,\quad
	\frac{\mathcal{L}_{2, \alpha}( p,  q)}{ q}=\frac{-d
		\Delta\Psi}{\Psi}-\zeta(x).
	\end{align*}
	Then, it is immediate to see that
	\begin{align*}
	\lambda_{R,\zeta}(\alpha)=\max_{(\Phi,\Psi)\in \Sigma'_\alpha}\min\bigg\{\inf_{\mathbb{R}}\frac{-D\Phi''-\nu\Psi(x,0)}{\Phi}+\mu,\  \inf_{\mathbb{R}\times[0,R)}\frac{-d
		\Delta\Psi}{\Psi}-\zeta(x)\bigg\},
	\end{align*}
	where $\Sigma'_\alpha:=\left\{(\Phi,\Psi)\in C^2(\mathbb{R})\times C^2(\overline\Omega_R):  ~  e^{\alpha x}(\Phi,\Psi)\in\Sigma_\alpha\right\}.$
	Let $\alpha_1$, $\alpha_2$ be real numbers and $t\in[0,1]$. Set $\alpha=t\alpha_1+(1-t)\alpha_2$. One  has to show that $\lambda_{R,\zeta}(\alpha)\ge t\lambda_{R,\zeta}(\alpha_1)+(1-t)\lambda_{R,\zeta}(\alpha_2)$. Let $(\Phi_1,\Psi_1)$ and $(\Phi_2,\Psi_2)$ be two arbitrarily chosen function  pairs in $\Sigma'_{\alpha_1}$ and $\Sigma'_{\alpha_2}$, respectively. Set $(w_1,z_1)=(\ln\Phi_1,\ln\Psi_1)$, $(w_2,z_2)=(\ln\Phi_2,\ln\Psi_2)$, $w=tw_1+(1-t)w_2$, $z=tz_1+(1-t)z_2$ and $(\Phi,\Psi)=(e^w, e^z)$. It follows that $(\Phi,\Psi)\in\Sigma'_\alpha$. Then, it is  obvious to see that
	\begin{equation*}
	\lambda_{R,\zeta}(\alpha)\ge\min\bigg\{\inf_\mathbb{R}\frac{-D\Phi''-\nu\Psi(x,0)}{\Phi}+\mu, \inf_{\mathbb{R}\times[0,R)}\frac{-d\Delta\Psi}{\Psi}-\zeta(x)\bigg\}.
	\end{equation*}
	After some calculations, one has
	\begin{align*}
	&\frac{-D\Phi''-\nu\Psi(x,0)}{\Phi}=-Dw''-Dw'^2-\nu e^{z(x,0)-w(x)},\ \ \ 
	\frac{-d\Delta\Psi}{\Psi}=-d\Delta z-d\nabla z\cdot\nabla z.
	\end{align*}
Noticing that $x\mapsto e^x$ is convex, $\nu>0$ and $t(1-t)\ge 0$, it follows that
	\begin{align*}
	\frac{-D\Phi''-\nu\Psi(x,0)}{\Phi}+\mu&\ge t\big(-Dw''_1-Dw'^2_1-\nu e^{z_1(x,0)-w_1}\big)\cr &~~~~~~~~~~~~~~~~~~~~~~~~~~~~~~+(1-t)\big(-Dw''_2-Dw'^2_2-\nu e^{z_2(x,0)-w_2}\big)+\mu\cr
	&\ge t\bigg(	\frac{-D\Phi''_1-\nu\Psi_1(x,0)}{\Phi_1}+\mu\bigg)+(1-t)\bigg(	\frac{-D\Phi''_2-\nu\Psi_2(x,0)}{\Phi_2}+\mu\bigg).
	\end{align*}
	Similarly, 
	\begin{align*}
	\frac{-d\Delta\Psi}{\Psi}-\zeta(x)\ge t\bigg(\frac{-d\Delta\Psi_1}{\Psi_1}-\zeta(x)\bigg)+(1-t)\bigg(\frac{-d\Delta\Psi_2}{\Psi_2}-\zeta(x)\bigg).
	\end{align*}
	Therefore,  
	\begin{align*}
	\lambda_{R,\zeta}(\alpha)\ge   &t\min\bigg\{
	\inf \frac{-D\Phi''_1-\nu\Psi_1(x,0)}{\Phi_1}+\mu, \inf \frac{-d\Delta\Psi_1}{\Psi_1}-\zeta(x)
	\bigg\}\\
	&\qquad +(1-t)\min\bigg\{\inf \frac{-D\Phi''_2-\nu\Psi_2(x,0)}{\Phi_2}+\mu, \inf \frac{-d\Delta\Psi_2}{\Psi_2}-\zeta(x)\bigg\}.
	\end{align*}
	Since $(\Phi_1,\Psi_1)$ and $(\Phi_2,\Psi_2)$ were arbitrarily chosen, one concludes that $\lambda_{R,\zeta}(\alpha)\ge t\lambda_{R,\zeta}(\alpha_1)+(1-t)\lambda_{R,\zeta}(\alpha_2)$. That is, $\alpha\mapsto\lambda_{R,\zeta}(\alpha)$ is concave in $\mathbb{R}$ and then continuous in $\mathbb{R}$.
	
	\textit{Step 6. The upper and lower bounds \eqref{bound'} of $\lambda_{R,\zeta}(\alpha)$.}  From Step 1, it follows that $\lambda^*_{R,\zeta}(\alpha)$ is positive, whence it is immediate to see that $\lambda_{R,\zeta}(\alpha)=\lambda^*_{R,\zeta}(\alpha)-\Lambda_\zeta(\alpha)> -\Lambda_\zeta(\alpha)$, namely,
	\begin{align*}
	-\lambda_{R,\zeta}(\alpha)< \max\big\{D\alpha^2+\nu-\mu+{\mu\nu}/{d}, d(\alpha^2+1)+M\big\}.
	\end{align*}
	It suffices to show that 
	\begin{align*}
	-\lambda_{R,\zeta}(\alpha)>\max\bigg\{D\alpha^2-\mu, d\alpha^2+m-d\frac{\pi^2}{R^2}\bigg\}.
	\end{align*} 
  From Step 3 we have that $-\lambda_{R,\zeta}(\alpha)$ is non-decreasing with respect to $\zeta$ for all $\alpha\in\mathbb{R}$,  it then follows that $-\lambda_{R,\zeta}(\alpha)\ge -\lambda_{R,m}(\alpha)$ for all $\alpha\in\mathbb{R}$. We claim that
	\begin{align}
	\label{sub-bound}
	-\lambda_{R,m}(\alpha)>\max\bigg\{D\alpha^2-\mu, d\alpha^2+m-d\frac{\pi^2}{R^2}\bigg\}.
	\end{align} 
	Inspired from  \cite[Proposition 3.4]{GMZ2015}, we assume by contradiction that $-D\alpha^2+\mu-\lambda_{R,m}(\alpha)\le 0$.
 Denote by $\big(\lambda_{R,m}(\alpha),(\tilde p,\tilde q)\big)$ the  principal eigenpair of eigenvalue problem \eqref{5.15/5.18'}  with $\zeta$ replaced by $m$, then $\big(\lambda_{R,m}(\alpha),(\tilde p,\tilde q)\big)$ satisfies 
\begin{align}
\label{4.4'}
\begin{cases}
-D\tilde p''+2D\alpha \tilde p'+(-D\alpha^2+\mu)\tilde p-\nu \tilde q(x,0)=\lambda_{R,m}(\alpha) \tilde p,  &x\in\mathbb{R},\\
-d\Delta\tilde q+2d\alpha\partial_x \tilde q-(d\alpha^2+m)\tilde q=\lambda_{R,m}(\alpha) \tilde q,  &(x,y)\in\Omega_R,\\
-d\partial_y\tilde q(x,0)+\nu \tilde q(x,0)-\mu \tilde p=0, \   &x\in\mathbb{R},\\
 \tilde q(x,R)=0, \   &x\in\mathbb{R},\\
\tilde p,  \tilde q \ \text{are} \ L\text{-periodic with respect to} \ x.
\end{cases}
\end{align} 
Since $\tilde p$ satisfies 
	\begin{equation}\label{z'}
	-D\tilde p''+2D\alpha \tilde p'+\big(-D\alpha^2+\mu-\lambda_{R,m}(\alpha)\big)\tilde p=\nu \tilde q(\cdot,0)>0~~~\text{in}~\mathbb{R},
	\end{equation}
 one infers that any positive constant is a subsolution of \eqref{z'}. Since $\tilde p$ is $L$-periodic in $x$, one gets that $\tilde p$ is identically equal to its minimum and thus $\tilde p$ is a positive constant in $\mathbb{R}$. Then, $0<\nu \tilde q(\cdot,0)=(-D\alpha^2+\mu-\lambda_{R,m}(\alpha))\tilde p\le 0$ in $\mathbb{R}$. This is a contradiction. Therefore, $-D\alpha^2+\mu-\lambda_{R,m}(\alpha)> 0$.
	
	Next, we assume that $\lambda_{R,m}(\alpha)\ge-d\alpha^2-m+d\frac{\pi^2}{R^2}$. We denote $w_R=\frac{\pi}{R}$, then
	\begin{equation*}
	w:=\sqrt{\frac{d\alpha^2+m+\lambda_{R,m}(\alpha)}{d}}\ge w_R>0.
	\end{equation*}
	Integrating the second equation in \eqref{4.4'} with respect to $x$ over $[0,L]$, then $\Psi(y):=\int_0^L\tilde q(x,y)\mathrm{d}x$ satisfies $\Psi''(y)+w^2\Psi(y)=0$,
	with $\Psi(y)>0$ in $[0,R)$, $\Psi(R)=0$. One gets that $\Psi(\cdot)=C\sin(w(R-\cdot))$ in $[0,R]$ for some constant $C>0$. Since $w\ge w_R$, it is easy to see that $[0,R)$ contains at least a half period of $\Psi$, namely, $\Psi$ must attain a non-positive value in $[0,R)$, which is impossible. Therefore, $\lambda_{R,m}(\alpha)<-d\alpha^2-m+d\frac{\pi^2}{R^2}$, namely, \eqref{sub-bound} is proved.
This completes the proof of (iv).
\end{proof}

In what follows, we shall give the variational formulas for $c^*_{R,\pm}$ by linear operators approach. For simplicity of the notation, we write $\lambda_R(\alpha):=\lambda_{R,\zeta}(\alpha)$ in the sequel. We have:
\begin{proposition}
	\label{formula}
	Let $c^*_{R,+}$ and $c^*_{R,+}$ be the rightward and leftward asymptotic spreading speeds of $Q_1$. Then,
	\begin{equation*}
	c^*_{R,+}=\inf\limits_{\alpha>0}\frac{-\lambda_{R}(\alpha)}{\alpha},~~~c^*_{R,-}=\inf\limits_{\alpha>0}\frac{-\lambda_{R}(-\alpha)}{\alpha}.
	\end{equation*}
\end{proposition} 
\begin{proof} 
	Since $f(x,v)\le f_v(x,0)v$ for all $(x,y)\in\overline{\Omega}_R$ and $v\ge 0$, it  follows that, for any $(u_0,v_0)\in \mathcal{K}$, the solution $(u(t,\cdot;u_0), v(t,\cdot,\cdot;v_0))$ of \eqref{pb in strip} is a strict subsolution of \eqref{linear pb in strip} for all $t>0$ and $(x,y)\in\overline{\Omega}_R$. By a  comparison argument, it implies that $Q_t[(u_0,v_0)]\le L(t)[(u_0,v_0)]$ for all $t> 0$ and $(u_0,v_0)\in\mathcal{K}$. Letting $t=1$, we have $Q_1((u_0,v_0))\le L(1)[(u_0,v_0)]$ for every $(u_0,v_0)\in\mathcal{K}$.
	
    Define a linear operator $\mathbb{L}_\alpha$ on $\mathbb{P}=\{(u,v)\in\mathbb{C}^1_0: (u_0,v_0)~\text{is}~L\text{-periodic in}~x\}$ associated with $L(1)$ by 
	\begin{align*}
	\mathbb{L}_\alpha[(u_0,v_0)]:&=e^{\alpha x}\cdot  L(1)[e^{-\alpha x}(u_0,v_0)]\\
	&=e^{\alpha x}\cdot e^{-\alpha x}L_\alpha(1)[(u_0,v_0)]\\
	&=L_\alpha(1)[(u_0,v_0)]~~~~ \text{for every}~ (u_0, v_0)\in \mathbb{P}~\text{and}~(x,y)\in\overline\Omega_R.
	\end{align*}  
	It then follows that $\mathbb{L}_\alpha=L_\alpha(1)$, and hence, $e^{-\lambda_R(\alpha)}$ is the principal eigenvalue of $\mathbb{L}_\alpha$. Since the function $\alpha\mapsto \ln ( e^{-\lambda_R(\alpha)})=-\lambda_R(\alpha)$ is convex, using similar arguments as in \cite[Theorem 2.5]{Weinberger2002} and \cite[Theorem 3.10(i)]{LZ2007}, we obtain that
	\begin{equation}
	\label{7.8}
	c^*_{R,+}\le \inf\limits_{\alpha>0}\frac{\ln( e^{-\lambda_R(\alpha)})}{\alpha}=\inf\limits_{\alpha>0}\frac{-\lambda_R(\alpha)}{\alpha}.
	\end{equation}
	On the other hand, for any given $\varep\in(0,1)$, there exists $\delta>0$ such that $f(x,v)\ge (1-\varep)f_v(x,0)v$ for all $v\in[0,\delta]$ and $(x,y)\in\overline\Omega_R$.
	By the continuity of the solutions of \eqref{pb in strip} with respect to the initial conditions given in Proposition \ref{prop3.8}, there exists a $L$-periodic (in $x$) positive function pair $(u_1,v_1)\in \text{Int}(\mathbb{P}^+)$ satisfying $u_1\le U_R$ in $\mathbb{R}$ and $v_1\le V_R$ in $\overline\Omega_R$ such that $u(t,x;u_1)\le \nu\delta/\mu,v(t,x,y;v_1)\le \delta$ for all $t\in[0,1]$ and $(x,y)\in\overline\Omega_R$.
	By Proposition \ref{cp-strip}, one infers that, for all $(u_0,v_0)\in\mathcal{K}_1:=\{(u,v)\in C(\mathbb{R})\times C(\overline\Omega_R): (0,0)\le (u,v)\le (u_1,v_1)~\text{in}~\overline\Omega_R\}$,
	$$u(t,\cdot;u_0)\le u(t,\cdot;u_1)\le \nu\delta/\mu~~ \text{for all}~ t\in[0,1]~\text{and}~ x\in\mathbb{R},~~~~~$$
	$$v(t,\cdot,\cdot;v_0)\le v(t,\cdot,\cdot;v_1)\le \delta~~ \text{for all}~ t\in[0,1]~\text{and}~ (x,y)\in\overline\Omega_R.$$
 Thus, for any $(u_0,v_0)\in\mathcal{K}_1$, the solution $(u(t,\cdot;u_0), v(t,\cdot,\cdot; v_0))$	of \eqref{pb in strip} satisfies
 \begin{align*}
 \begin{cases}
 u_t-Du_{xx}=\nu v(t,x,0)-\mu u,\quad &t\in[0,1],\ x\in\mathbb{R}, \\
 v_t-d\Delta v\ge (1-\varep)f_v(x,0)v,\quad &t\in[0,1],\ (x,y)\in\Omega_R, \\
 -d\partial_y v(t,x,0)=\mu u-\nu v(t,x,0),\quad &t\in[0,1],\ x\in\mathbb{R}, \\
 v(t,x,R)=0,\quad &t\in[0,1],\ x\in\mathbb{R}.
 \end{cases}
 \end{align*}

Let $\{\mathbb{L}^\varep(t)\}_{t\ge 0}$ be the solution semigroup generated by the following linear system:
\begin{align*}
\begin{cases}
 u_t-Du_{xx}=\nu v(t,x,0)-\mu u,\quad &t>0,\ x\in\mathbb{R}, \\
v_t-d\Delta v= (1-\varep)f_v(x,0)v,\quad &t>0,\ (x,y)\in\Omega_R, \\
-d\partial_y v(t,x,0)=\mu u-\nu v(t,x,0),\quad &t>0,\ x\in\mathbb{R}, \\
v(t,x,R)=0,\quad &t>0,\ x\in\mathbb{R}.
\end{cases}
\end{align*}
Then, Proposition \ref{cp-strip} implies that $\mathbb{L}^\varep(t)[(u_0,v_0)]\le Q_t[(u_0,v_0)]$ for all $t\in[0,1]$ and  $(u_0,v_0)\in \mathcal{K}_1$. In particular, $\mathbb{L}^\varep(1)[(u_0,v_0)]\le Q_1[(u_0,v_0)]$ for all  $(u_0,v_0)\in \mathcal{K}_1$.
	
	Let $\lambda^\varep_R(\alpha)$ be the  principal eigenvalue of the eigenvalue problem \eqref{5.15/5.18'} with $f_v(x,0)$ replaced by $(1-\varep)f_v(x,0)$. As argued above, the concavity of $\lambda^\varep_R(\alpha)$ and  similar arguments as in \cite[Theorem 2.4]{Weinberger2002} and \cite[Theorem 3.10(ii)]{LZ2007} give rise to 
	\begin{equation}
	\label{7.11}
	c^*_{R,+}\ge \inf\limits_{\alpha>0}\frac{\ln( e^{-\lambda^\varep_R(\alpha)})}{\alpha}=\inf\limits_{\alpha>0}\frac{-\lambda^\varep_R(\alpha)}{\alpha}~~\text{for all}~ \varep\in (0,1).
	\end{equation}
	Combining \eqref{7.8} and \eqref{7.11}, we obtain
	\begin{equation*}
	\inf\limits_{\alpha>0}\frac{-\lambda^\varep_R(\alpha)}{\alpha}\le c^*_{R,+}\le \inf\limits_{\alpha>0}\frac{-\lambda_R(\alpha)}{\alpha}~~\text{for all}~ \varep\in (0,1).
	\end{equation*}
Letting $\varep\to 0$, thanks to the continuity of the function  $\zeta\mapsto\lambda_{R,\zeta}(\alpha)$ in Proposition \ref{principal eigenvalue in strip} (ii), we then have $$c^*_{R,+}=\inf\limits_{\alpha>0}\frac{-\lambda_R(\alpha)}{\alpha}.$$

By change of variables $\hat u(t,x):=u(t,-x)$ and $\hat v(t,x,y):=v(t,-x,y)$, it follows that $c^*_{R,-}$ is the rightward asymptotic spreading speed of the resulting system for $(\hat u,\hat v)$. From the lines as above, it can be derived that $$c^*_{R,-}=\inf_{\alpha>0}\frac{-\lambda_R(-\alpha)}{\alpha}.$$ The proposition is therefore proved.
\end{proof}
 
\begin{lemma}\label{lemma-same speed}
	$c^*_{R,+}=c^*_{R,-}>0.$
\end{lemma}
\begin{proof}
 We first prove that	$c^*_{R,+}=c^*_{R,-}$. By virtue of the variational formulas obtained above, it is enough to show  $\lambda_R(\alpha)=\lambda_R(-\alpha)$. 
	Let $(\lambda_R(\alpha);(p,q))$ be the principal eigenpair of the eigenvalue problem \eqref{5.15/5.18'}, namely,
	\begin{align}
	\label{right}
	\begin{cases}
	-Dp''+2D\alpha p'+(-D\alpha^2+\mu) p-\nu q(x,0)=\lambda_R(\alpha) p,  &x\in\mathbb{R},\\
	-d\Delta q+2d\alpha\partial_x q-(d\alpha^2+f_v(x,0)) q=\lambda_R(\alpha) q,  &(x,y)\in\Omega_R,\\
	-d\partial_y q(x,0)+\nu q(x,0)-\mu p=0, \   &x\in\mathbb{R},\\
	q(x,R)=0, \   &x\in\mathbb{R},\\
	p,  q \ \text{are} \ L\text{-periodic with respect to} \ x,
	\end{cases}
	\end{align}
	and let $(\lambda_R(-\alpha);(\phi,\psi))$ be the principal eigenpair of the eigenvalue problem \eqref{5.15/5.18'}, that is,
		\begin{align}
	\label{left}
	\begin{cases}
	-D\phi''-2D\alpha \phi'+(-D\alpha^2+\mu) \phi-\nu \psi(x,0)=\lambda_R(-\alpha) \phi,  &x\in\mathbb{R},\\
    -d\Delta \psi-2d\alpha\partial_x \psi-(d\alpha^2+f_v(x,0)) \psi=\lambda_R(-\alpha) \psi,  &(x,y)\in\Omega_R,\\
	-d\partial_y \psi(x,0)+\nu \psi(x,0)-\mu \phi=0, \   &x\in\mathbb{R},\\
	\psi(x,R)=0, \   &x\in\mathbb{R},\\
	\phi,  \psi \ \text{are} \ L\text{-periodic with respect to} \ x.
	\end{cases}
	\end{align}
	We multiply the first equations in \eqref{right}  and  in \eqref{left} by  $\phi$ and $p$, repectively, then we integrate the two resulting equations over $(0,L)$. By  subtraction, it follows that
	\begin{equation*}
	\big[\lambda_R(\alpha)-\lambda_R(-\alpha)\big]\int_0^L p\phi \mathrm{d}x=-\nu\int_0^L\big(q(x,0)\phi-\psi(x,0)p\big) \mathrm{d}x.
	\end{equation*}
	Similarly, we multiply the second equations in \eqref{right}  and in \eqref{left} by $\psi$ and $q$, respectively. By subtracting the integration of the two resulting equations over $S=(0,L)\times (0,R)$, one gets
	\begin{equation*}
	\big[\lambda_R(\alpha)-\lambda_R(-\alpha)\big]\int_S q\psi \mathrm{d}x \mathrm{d}y=\mu\int_0^L\big(q(x,0)\phi-\psi(x,0)p\big)\mathrm{d}x.
	\end{equation*}
	Therefore, by using the positivity of $(p,q)$ and $(\phi,\psi)$, one has
	\begin{equation*}
	\text{sgn}\big(\lambda_R(\alpha)-\lambda_R(-\alpha)\big)=\text{sgn} \Big(\int_0^L\big(q(x,0)\phi-\psi(x,0)p\big)\mathrm{d}x\Big) =-\text{sgn}\Big(\int_0^L\big(q(x,0)\phi-\psi(x,0)p\big)\mathrm{d}x\Big),
	\end{equation*}
	which implies that $\lambda_R(\alpha)=\lambda_R(-\alpha)$. Consequently, $c^*_{R,+}=c^*_{R,-}$.
	
 From $\lambda_R(\alpha)=\lambda_R(-\alpha)$ and from Proposition \ref{principal eigenvalue in strip} (iv), it is seen that the function $\alpha\mapsto-\lambda_R(\alpha)$ is convex and even in $\mathbb{R}$ and $-\lambda_R(0)\ge m-d\pi^2/(R^2)>0$. Thus, $-\lambda_R(\alpha)>0$ for all $\alpha\in\mathbb{R}$, whence $c^*_{R,+}=c^*_{R,-}> 0$.
\end{proof}

 \begin{proof}[Proofs of Theorems \ref{thm-asp-strip} and \ref{thm-PTF-in strip}] By Theorems 3.4, 4.3 and 4.4 in \cite{LZ2010}, as well as Lemma \ref{lemma-same speed} above, one derives the conclusion of Theorem \ref{thm-asp-strip} with spreading speed $c^*_R$, as well as   the existence of the non-increasing in $s$ rightward and non-decreasing in $s$ leftward periodic traveling waves
 for problem \eqref{pb in strip} with minimal wave speed $c^*_R$. To complete the proof of Theorem \ref{thm-PTF-in strip}, it remains to show that these periodic traveling fronts are strictly monotone in $s$. For $c\ge c^*_R$, consider a periodic rightward traveling front of \eqref{pb in strip} (the case of leftward waves can be dealt with similarly), written as $(\phi_R(s,x),\psi_R(s,x,y))=(u(\frac{x-s}{c},x),v(\frac{x-s}{c},x,y))$ for all $s\in\mathbb{R}$ and $(x,y)\in\overline\Omega_R$. Notice that $(u(t,x),v(t,x,y))$  satisfies \eqref{pb in strip} and \eqref{periodicity-strip}, and is defined for all $t\in\mathbb{R}$ and $(x,y)\in\overline\Omega_R$. Since $c\ge c^*_R>0$, the function pair $(u,v)$  is non-decreasing in $t\in\mathbb{R}$. Then, for any $\tau>0$, $w(t,x)=u(t+\tau,x)-u(t,x)\ge 0$, $z(t,x,y)=v(t+\tau,x,y)-v(t,x,y)\ge 0$  for all $t\in\mathbb{R}$ and $(x,y)\in\overline\Omega_R$. The function pair $(w,z)$ is a classical solution to a linear problem in $\mathbb{R}\times\overline\Omega_R$.
 The strong parabolic maximum principle and the Hopf lemma, as well as the uniqueness of the corresponding Cauchy problem then imply that either $(w,z)$ is indentically $(0,0)$ or positive everywhere in $\mathbb{R}\times[0,R)$. If $(w,z)\equiv(0,0)$, then $(\phi_R(s-c\tau,x),\psi_R(s-c\tau,x,y))=(\phi_R(s,x),\psi_R(s,x,y))$ for all $s\in\mathbb{R}$ and $(x,y)\in\overline\Omega_R$, which contradicts the limit condition \eqref{limit-strip} as $s\to\pm\infty$ due to $c\tau>0$. Therefore, $w>0$ in $\mathbb{R}$ and $z>0$ in $\mathbb{R}\times[0,R)$ for any $\tau>0$. Hence, $(\phi_R(s,x),\psi_R(s,x,y))$ is decreasing in $s$. This completes the proof. 
\end{proof}

\section{Propagation properties in the half-plane: Proofs of Theorems \ref{thm-asp-half-plane} and \ref{PTF-in half-plane}}
\label{section5}
This section is devoted to  propagation properties for problem \eqref{pb in half-plane} in the half-plane. We only sketch the detailed proof in the right direction along the road, since the discussion in the left direction can be handled similarly. 

\subsection{The generalized eigenvalue problem in the half-plane}\label{generalized p-e}
Recall from Proposition \ref{principal eigenvalue in strip} that
\begin{equation}
\label{bound in strip}
\max\Big\{D\alpha^2-\mu, d\alpha^2+m-d\frac{\pi^2}{R^2}\Big\}<-\lambda_R(\alpha)<
\max\Big\{D\alpha^2+\nu-\mu+\frac{\mu\nu}{d}, d(\alpha^2+1)+M\Big\},
\end{equation} 
and the function $R\mapsto -\lambda_R(\alpha)$ is increasing. For any fixed $\alpha\in\mathbb{R}$, we can take the limit as follows:
\begin{equation}
\label{principal-e}
\lambda(\alpha):=\lim\limits_{R\to+\infty} \lambda_R(\alpha).
\end{equation}
 It can be deduced from \eqref{bound in strip} that
\begin{equation}
\label{bound in half-plane}
\max\Big\{D\alpha^2-\mu, d\alpha^2+m\Big\}\le -\lambda(\alpha)\le 
\max\Big\{D\alpha^2+\nu-\mu+\frac{\mu\nu}{d}, d(\alpha^2+1)+M\Big\}.
\end{equation}
 Since the function $\alpha\mapsto -\lambda_R(\alpha)$ is convex and continuous in $\mathbb{R}$ and since the pointwise limit of a convex function is still convex, it follows  that the function $\alpha\mapsto -\lambda(\alpha)$ is convex and continuous in $\mathbb{R}$. Furthermore, we have:
\begin{theorem}
	\label{thm 5.1}
	For any $\alpha\in\mathbb{R}$, let  $\lambda(\alpha)$ be defined by \eqref{principal-e}. Then there exists a positive $L$-periodic (in $x$) function pair  $(P_\alpha(x), Q_\alpha(x,y))$ associated with $\varLambda=\lambda(\alpha)$ satisfying
	\begin{align}
	\label{eigenvalue pb in half-plane}
	\begin{cases}
	-DP_\alpha''+2D\alpha P_\alpha'+(-D\alpha^2+\mu) P_\alpha-\nu Q_\alpha(x,0)=\varLambda P_\alpha,  &x\in\mathbb{R},\\
	-d\Delta Q_\alpha+2d\alpha\partial_x Q_\alpha-(d\alpha^2+f_v(x,0)) Q_\alpha=\varLambda Q_\alpha,  &(x,y)\in\Omega,\\
	-d\partial_y Q_\alpha(x,0)+\nu Q_\alpha(x,0)-\mu P_\alpha=0, \   &x\in\mathbb{R},\\
	 P_\alpha,  Q_\alpha \ \text{are positive and} \ L\text{-periodic with respect to} \ x,
	\end{cases}
	\end{align}
	and such that, up to some normalization,
	\begin{equation*}
	P_\alpha\le 1~\text{in}~\mathbb{R},~~Q_\alpha~\text{is locally bounded in}~\overline\Omega.
	\end{equation*}
	We call $\lambda(\alpha)$ the generalized principal eigenvalue of \eqref{eigenvalue pb in half-plane} and  $(P_\alpha, Q_\alpha)$ the
generalized principal eigenfunction pair  associated with $\lambda(\alpha)$.
Moreover,  problem \eqref{eigenvalue pb in half-plane} admits no positive and $L$-periodic (in $x$) eigenfunction pair for any $\varLambda>\lambda(\alpha)$.
\end{theorem}
 
 \begin{remark}
 	\textnormal{We point out here that the classical Krein-Rutman theorem cannot be applied anymore due to the noncompactness of the domain.  We denote by $(P_R,Q_R):=(P_{\alpha,R},Q_{\alpha,R})$  the principal eigenfunction pair of \eqref{eigen-pb-strip} in $\overline\Omega_R$ associated with the principal eigenvalue $\lambda_R(\alpha)$ for simplicity. As  will be shown later, with the technical Lemmas \ref{lemma3.5}--\ref{lemma3.7}, we can show that, up to normalization,  $\lim_{R\to +\infty}(P_{R},Q_{R})$ turns out to be the generalized principal eigenfunction pair $(P_\alpha, Q_\alpha)$ of \eqref{eigenvalue pb in half-plane} in $\overline\Omega$ corresponding to the generalized principal eigenvalue $\lambda(\alpha)$. The statements of Lemmas \ref{lemma3.5}--\ref{lemma3.7} are similar to Lemmas 3.5--3.7 in \cite{GMZ2015}, however, our case is much more involved, since  the heterogeneous assumption is now set on $f$, this does not allow us to get the nice upper estimate as in Lemma 3.6 of \cite{GMZ2015}. 
 	For the sake of completeness, we give the details below.
 	} 
 \end{remark}
  
\begin{lemma}
	\label{lemma3.5}
	For any $R>R_0$, normalizing with $\Vert P_R(\cdot)\Vert_{L^\infty(\mathbb{R})}=1$, there exists $C_1>0$ (independent of $R$) such that
	\begin{equation*}
	\Vert Q_R(\cdot,0)\Vert_{L^\infty(\mathbb{R})}>C_1.
	\end{equation*}
\end{lemma}
\begin{proof}
	If the conclusion is not true, we assume  that  there exists a sequence $(R_k)_{k\in\mathbb{N}}$ satisfying $R_k\to+\infty$ such that $\Vert Q_{R_k}(\cdot,0)\Vert_{L^{\infty}(\mathbb{R})}\to 0$ and $\Vert P_{R_k}\Vert_{L^{\infty}(\mathbb{R})}=1$. Since $(P_{R_k},Q_{R_k})$ is $L$-periodic in $x$, we assume with no loss of generality that $x_k\in [0,L]$ such that $P(x_k)=1$. Since $(P_{R_k})_{k\in\mathbb{N}}$ and $(Q_{R_k}(\cdot, 0)
	)_{k\in\mathbb{N}}$ are uniformly bounded, by the Arzel\`{a}-Ascoli Theorem, up to extraction of a subsequence, one has $P_{R_k}\to P_{\infty}\ge 0$ and  $Q_{R_k}(\cdot, 0)\to 0$ as $k\to+\infty$. Moreover, there exists $x_\infty\in [0,L]$ such that, up to a subsequence,  $x_k\to x_\infty$ as $k\to+\infty$. Passing to the limit $k\to+\infty$ in the first equation of eigenvalue problem \eqref{eigen-pb-strip} satisfied by $(P_{R_k},Q_{R_k})$ in $\overline\Omega_{R_k}$ implies 
	\begin{equation*}
	-DP''_\infty+2D\alpha P'_\infty+(-D\alpha^2+\mu)P_\infty=\lambda(\alpha)P_\infty~~\text{in}~\mathbb{R}.
	\end{equation*}
	Moreover, $P_\infty$ is $L$-periodic in $x$ and $P_\infty(x_\infty)=1$. The strong maximum principle implies $P_\infty>0$ in $\mathbb{R}$. Thus, $P_\infty$ is a positive constant. Hence, $\lambda(\alpha)=-D\alpha^2+\mu$. This implies  $\lambda_R(\alpha)\ge \lambda(\alpha)= -D\alpha^2+\mu$
	, which contradicts  \eqref{bound in strip}. Consequently, Lemma \ref{lemma3.5} is proved.
\end{proof}
\begin{lemma}
	\label{lemma3.6}
	For any $R>R_0$, assume that $\Vert Q_R(\cdot, 0)\Vert_{L^\infty(\mathbb{R})}=1$, then  $Q_R(x,y)$ is locally bounded   as $R\to+\infty$ by some positive constant (independent of $R$). 
\end{lemma}
\begin{proof}
	For convenience, let us introduce some new notations. For $n>R_0$ large enough, we denote by $(\lambda_n(\alpha);(P_n,Q_n))$ the  principal eigenpair of \eqref{eigen-pb-strip}  in $\overline\Omega_n=\mathbb{R}\times[0,n]$ with normalization $\Vert Q_n(\cdot, 0)\Vert_{L^\infty(\mathbb{R})}=1$. Then, one has to show that, for any compact set $K\subset\overline\Omega$, there holds
	\begin{equation}\label{lem3.6-1}
	\sup\limits_n(\max\limits_{K\cap\overline\Omega_n}Q_n(x,y))<+\infty.
	\end{equation}
 To prove this, we first claim that 
	 $\Vert P_n\Vert_{L^\infty(\mathbb{R})}\le C_0$ for some constant $C_0>0$. Assume by contradiction that $\Vert P_n\Vert_{L^\infty(\mathbb{R})}$ is unbounded, then we  choose a sequence $(P_n)_{n\in\mathbb{N}}$ such that $\Vert P_n\Vert_{L^\infty(\mathbb{R})}\to +\infty$ as $n\to +\infty$. By renormalization, it follows that $\Vert P_n\Vert_{L^\infty(\mathbb{R})}=1$ while $\Vert Q_n(\cdot,0)\Vert_{L^\infty(\mathbb{R})}\to 0$. This contradicts the conclusion of Lemma \ref{lemma3.5}. Our claim is thereby achieved.
	 It then follows from the boundary condition $-d\partial_y Q_n(\cdot, 0)=\mu P_n(\cdot)-\nu Q_n(\cdot,0)$ that $\Vert\partial_y Q_n(\cdot,0)\Vert_{L^\infty(\mathbb{R})}\le (\mu C_0+\nu)/d$. Assume now that \eqref{lem3.6-1} is not true. Then,  there exist a compact subset $K\subset\overline\Omega$ and a sequence $(x_n,y_n)_{n\in\mathbb{N}}$ in $K\cap\overline\Omega_n$ so that $Q_n(x_n,y_n)=\max_{K\cap\overline\Omega_n} Q_n>n$. Then we are able to find a larger compact set containing $K$ such that this assumption is still satisfied. Therefore,    without loss of generality we take $K=\overline{B^+_\rho((0,0))}$ with radius $\rho$ large. Therefore, up to extraction of some subsequence, $x_n\to x_\infty\in[-\rho,\rho]$, $y_n\to y_\infty\in [0,+\infty)$ as $n\to +\infty$, thanks to the boundedness of $(y_n)_{n\in\mathbb{N}}$.  It follows that
	 either $y_\infty>0$ or $ y_\infty=0$.
	  By setting 
	\begin{equation*}
	w_n(x,y):=\frac{Q_n(x,y)}{Q_n(x_n,y_n)}~~\text{in}\  K\cap\overline\Omega_n,
	\end{equation*}
	one has $0<w_n\le 1$ in $K\cap\overline\Omega_n$ and $w_n(\cdot,0)<\frac{1}{n}$ in $[-\rho,\rho]$ for all $n$ large enough. In particular,  $w_n(x_n,y_n)=1$. It can be deduced that the function $w_n$ satisfies
	\begin{equation*}
	\begin{cases}
	-d \Delta w_n+2d\alpha\partial_x w_n-(d\alpha^2+f_v(x,0)+\lambda_n(\alpha))w_n=0,\quad &\text{in}\  K\cap\overline\Omega_n,\\
	-d\partial_y w_n(x,0)=\mu\frac{P_n(x)}{Q_n(x_n,y_n)}-\nu w_n(x,0), \quad &\text{in}\  [-\rho,\rho].
	\end{cases}
	\end{equation*}
	From standard elliptic estimates up to the boundary, the positive function $w_n$ converges, up to extraction of some subsequence,  to a classical solution $w_\infty\in [0,1]$  of 
	\begin{equation*}
	\begin{cases}
	-d\Delta w_\infty+2d\alpha\partial_x w_\infty-(d\alpha^2+f_v(x,0)+\lambda(\alpha))w_\infty=0, \quad &\text{in}\  K\cap\overline\Omega,\\
	-d\partial_y w_\infty(x,0)+\nu w_\infty (x,0)=0, \quad &\text{in}\  [-\rho,\rho].
	\end{cases}
	\end{equation*}
	Moreover, $w_\infty(\cdot,0)=0$ in $[-\rho,\rho]$ and $w_\infty(x_\infty,y_\infty)=1$. 
	Therefore,  the case that  $y_\infty=0$ is impossible. Assume now that $y_\infty>0$. By using the Harnack inequality up to the boundary, there exists a point $(x',y')$ in the neighborhood of $(x_\infty,y_\infty)$ belonging to $ (K\cap\overline\Omega)^\circ$ such that $w_\infty(x',y')\ge \frac{1}{2}$. Then, the strong maximum principle implies that $w_\infty>0$ in $(K\cap\overline\Omega)^\circ$.
	Since $w_\infty(\cdot,0)=0$ in $[-\rho,\rho]$, one infers from the boundary condition  that  $\partial_y w_\infty(\cdot,0)=0$ in $[-\rho,\rho]$.
	This is a contradiction with the Hopf lemma.  This completes the proof of Lemma \ref{lemma3.6}.
\end{proof}
\begin{lemma}
	\label{lemma3.7}
For any $R>R_0$, normalizing with $\Vert P_R(\cdot)\Vert_{L^\infty(\mathbb{R})}=1$,  there is  $C_2>0$ (independent of $R$) such that
\begin{equation*}
\Vert Q_R(\cdot,0)\Vert_{L^\infty(\mathbb{R})}\le C_2.
\end{equation*}	
\end{lemma}
\begin{proof}
	If the statement is not true, by suitable renormalization we assume that there is a sequence $(R_n)_{n\in\mathbb{N}}$ satisfying  $R_n\to +\infty$ such that $\Vert Q_{R_n}(\cdot, 0)\Vert_{L^\infty(\mathbb{R})}=1$ and such that $\Vert P_{R_n}(\cdot)\Vert_{L^\infty(\mathbb{R})}\to 0$. Without loss of generality, we assume that $x_n\in [0,L]$ for all $n\in\mathbb{N}$, such that $Q_{R_n}(x_n,0)=1$. Therefore, there is $x_\infty\in[0,L]$ such that, up to some subsequence,  $x_n\to x_\infty$ as $n\to+\infty$. Since $(P_{R_n})_{n\in\mathbb{N}}$ and $(Q_{R_n}(\cdot,0))_{n\in\mathbb{N}}$ are uniformly bounded in $L^\infty(\mathbb{R})$, it follows from Lemma \ref{lemma3.6} and from standard elliptic estimates up to the boundary that the function pair $(P_{R_n},Q_{R_n})$ converges as $n\to+\infty$, up to extraction of some subsequence, locally uniformly in $\overline\Omega$ to  $(P_\infty,Q_\infty)$. In particular, $P_\infty\equiv 0$ in $\mathbb{R}$ and $Q_\infty(x_\infty,0)=1$. Moreover, $P_\infty$ satisfies
	\begin{equation*}
	-DP_\infty''+2D\alpha P_\infty'+(-D\alpha^2+\mu)P_\infty-\nu Q_\infty(\cdot,0)=\lambda(\alpha)P_\infty~~\text{in}~\mathbb{R}.
	\end{equation*}
	Then, it is easily derived from above equation that $Q_\infty(\cdot,0)\equiv 0$ in $\mathbb{R}$, which contradicts $Q_\infty(x_\infty,0)=1$. The proof of this lemma is thereby complete.
\end{proof}
\begin{proof}[Proof of Theorem \ref{thm 5.1}]
By elliptic estimates and  Lemmas \ref{lemma3.5}--\ref{lemma3.7}, the eigenfunction pair $(P_R,Q_R)$ converges locally uniformly in $\overline\Omega$ as $R\to+\infty$ to a nonnegative and $L$-periodic (in $x$) function pair $(P_\alpha, Q_\alpha))$ solving the generalized eigenvalue problem \eqref{eigenvalue pb in half-plane} in the half-plane $\overline\Omega$ associated with the generalized principal eigenvalue $\lambda(\alpha)$. Moreover, up to normalization, it follows that $P_\alpha\le 1$ in $\mathbb{R}$ and $Q_\alpha$ is locally bounded in $\overline \Omega$. By the strong maximum principle and the Hopf Lemma, we further derive that $(P_\alpha, Q_\alpha)$ is positive in $\overline\Omega$.

Assume that $\varLambda$  corresponds to a positive and $L$-periodic (in $x$) eigenfunction pair $(P,Q)$ such that the generalized eigenvalue problem \eqref{eigenvalue pb in half-plane} is satisfied. By reasoning as in the proof of Proposition \ref{principal eigenvalue in strip} (iii), it follows that $\varLambda<\lambda_R(\alpha)$ for any $R>R_0$, which reveals  $\varLambda\le \lambda(\alpha)$ by taking $R\to+\infty$.	
\end{proof}
\subsection{Spreading speeds and pulsating fronts in the half-plane}
This subsection is devoted to the proofs of Theorems \ref{thm-asp-half-plane} and  \ref{PTF-in half-plane}. We start with   variational characterization of the rightward and leftward asymptotic spreading speeds $c^*_\pm$  by using the generalized principal eigenvalue  constructed in the preceding subsection. 
Define 
\begin{equation*}
c^*_+:=\inf_{\alpha>0}\frac{-\lambda(\alpha)}{\alpha},~~~c^*_-:=\inf_{\alpha>0}\frac{-\lambda(-\alpha)}{\alpha}.
\end{equation*}
Thanks to \eqref{bound in half-plane}, it is noticed that $c^*_+\in [2\sqrt{dm},+\infty)$ is well-defined. Moreover, we point out that, from the definitions of $\lambda(\pm\alpha)$ and of $c^*_{\pm}$ and from the property that $\lambda_{R}(\alpha)=\lambda_R(-\alpha)$ for all $\alpha\in\mathbb{R}$ (for any $R>R_0$) shown in the proof of Lemma \ref{lemma-same speed}, it is obvious to see that $c^*_+=c^*_-$. In what follows, we denote $c^*:=c^*_+=\inf_{\alpha>0}-\lambda(\alpha)/\alpha>0$.
\begin{lemma} 
	\label{lemma-c*}
	There holds $c^*_R<c^*$ and
	$c^*_R\to c^*$  as $R\to+\infty$.
\end{lemma}
\begin{proof}
	 Since the function $R\mapsto-\lambda_R(\alpha)$ is increasing  for all $\alpha\in\mathbb{R}$, one has $-\lambda_R(\alpha)<-\lambda(\alpha)$ for all  $\alpha\in\mathbb{R}$. This implies 	
	\begin{equation*}
	\frac{-\lambda_R(\alpha)}{\alpha}<\frac{-\lambda(\alpha)}{\alpha} ~~~ \text{for all}\  \alpha>0.
	\end{equation*}
	Furthermore,
	\begin{equation*}
	\inf\limits_{\alpha>0}\frac{-\lambda_R(\alpha)}{\alpha}<\inf\limits_{\alpha>0}\frac{-\lambda(\alpha)}{\alpha},
	\end{equation*}
	which implies 
	\begin{equation}
	\label{c*}
	0<c^*_R<c^*.
	\end{equation}
	
	It remains to prove that $c^*_R\to c^*$  as $R\to+\infty$. Since the functions $\alpha\mapsto
	-\lambda_R(\alpha)$ and  $\alpha\mapsto-\lambda(\alpha)$ are convex and continuous in $\mathbb{R}$, one has $\alpha\mapsto-\lambda_R(\alpha)/\alpha$ and $\alpha\mapsto-\lambda(\alpha)/\alpha$ 
	are continuous for all $\alpha\in(0,+\infty)$. Since 
	 $-\lambda_R(\alpha)/\alpha$ increasingly converges to $-\lambda(\alpha)/\alpha$ as $R\to+\infty$ for each $\alpha\in(0,+\infty)$, the  Dini's Theorem (see, e.g., \cite[Theorem 7.13]{Rudin1976}) implies that 
	\begin{equation*}
	\frac{-\lambda_R(\alpha)}{\alpha}\to \frac{-\lambda(\alpha)}{\alpha} \quad \text{as}\ R\to+\infty~~\text{uniformly in}~\alpha\in (0,+\infty).
	\end{equation*}
		On the other hand, it is seen from \eqref{bound in strip} and \eqref{bound in half-plane} that both $-\lambda_R(\alpha)/\alpha$ and $-\lambda(\alpha)/\alpha$ tend to infinity as $\alpha\to 0^+$ and as $\alpha\to+\infty$. One then concludes that
	\begin{equation*}
	\inf_{\alpha>0}\frac{-\lambda_R(\alpha)}{\alpha}\to \inf_{\alpha>0}\frac{-\lambda(\alpha)}{\alpha} \quad \text{as}\ R\to+\infty.
	\end{equation*}
	That is, $c^*_R\to c^*$ as $R\to+\infty$. The proof is thereby complete.
\end{proof}

 \begin{proof}[Proof of Theorem \ref{thm-asp-half-plane}]
    (i) We first construct the upper bound in the rightward propagation. 
   Let $(u,v)$ be the solution of \eqref{pb in half-plane} with nonnegative, bounded, continuous and compactly supported initial condition $(u_0,v_0)\not\equiv(0,0)$. We need to show
   \begin{equation}
   \label{find upper bound}
   \lim\limits_{ t\to+\infty}\sup\limits_{x\ge ct,\ 0\le y\le A} |(u(t,x),v(t,x,y))|=0~~\text{for all}~c> c^*,
   \end{equation}    
 For any $c>c^*$, choose  $c'\in[c^*,c)$ and $\alpha>0$ such that $-\lambda(
 	\alpha)=\alpha c'$. Let $(\lambda(\alpha);(P_\alpha, Q_\alpha))$ be the  generalized principal eigenpair of \eqref{eigenvalue pb in half-plane} derived in Theorem \ref{thm 5.1}. Since $(u_0,v_0)$ is compactly supported, 
 	 one infers that, 	
 	 for some $\gamma>0$, $\gamma e^{-\alpha(x-c't)}(P_\alpha(x), Q_\alpha(x,y))$ lies above $(u_0,v_0)$ at time $t=0$. Thanks to the KPP assumption, one further deduces that 
 	 $\gamma e^{-\alpha(x-c't)}(P_\alpha(x), Q_\alpha(x,y))$ 	
 	  is an exponential  supersolution of the Cauchy problem \eqref{pb in half-plane} and  $\gamma e^{-\alpha(x-c't)}(P_\alpha(x), Q_\alpha(x,y))\ge (u(t,x),v(t,x,y))$ for all $t\ge 0$ and $(x,y)\in\overline\Omega$ by Proposition \ref{cp}. It follows that, for any $A>0$,
 	\begin{equation*}
 	\sup\limits_{x\ge ct,0\le y\le A} (u(t,x),v(t,x,y))\le \sup\limits_{x\ge ct,0\le y\le A}\gamma e^{-\alpha(c-c')t}(P_\alpha(x),Q_\alpha(x,y)),
 	\end{equation*}
 	whence, by Theorem \ref{thm 5.1} and by passing to the limit $t\to+\infty$, the formula \eqref{find upper bound} is proved.  
 	
 	(ii) Let us  prove the lower bound \eqref{find lower bound}. 
 	Choose any $c\in(0,c^*)$. Let $(u,v)$ be the solution of \eqref{pb in half-plane} with nonnegative, nontrivial, bounded and continuous  initial condition $(u_0,v_0)<(\nu/\mu,1)$. Thanks to \eqref{truncated to half-plane}, we know that  $(U_B(x), V_B(x,y))$ increasingly converges to $(\nu/\mu,1)$ as $B\to +\infty$ uniformly in $x$ and locally uniformly in $y$.
 	Since $(u_0,v_0)<(\nu/\mu,1)$ in $\overline\Omega$, for $B>R_0$ sufficiently large, there is a smooth cut-off function $\chi^B:[0,+\infty)\mapsto[0,1]$ satisfying $\chi^B(\cdot)=1$ in $[0,B-1]$ and $\chi^B(\cdot)=0$ in $[B,+\infty)$, such that $(0,0)\le (u_0,\chi^B v_0)\le (U_B,V_B)$ in $\overline\Omega_B$. Let $(u_B,v_B)$ be the solution to the Cauchy problem \eqref{pb in strip} in $\overline\Omega_B$ with initial datum $(u_0,\chi^B v_0)$ and let  $(U_B, V_B)$ be  the associated unique nontrivial stationary solution of \eqref{pb in strip}. 	By  Lemma \ref{lemma-c*},  up to increasing $B$, the  asymptotic spreading speed $c^*_B$ of the solution $(u_B,v_B)$ to  \eqref{pb in strip} in $\overline\Omega_B$ can be very close to $c^*$, say $c^*_B\sim c^*$, such that  $c<c^*_B<c^*$. From Theorem \ref{thm-asp-strip}, one derives 
 	\begin{equation*}
 	\lim\limits_{t\to+\infty}\inf\limits_{0\le x\le ct,\ y\in [0,B]}(u_B(t,x),v_B(t,x,y))=(U_B(x), V_B(x,y)),
 	\end{equation*}
 due to  $0<c<c^*_B$. Notice that  $(u,v)$  is a strict supersolution to  problem \eqref{pb in strip} with initial datum $(u_0,\chi^Bv_0)$ in $\overline\Omega_B$,  Proposition \ref{cp-strip} yields $(u(t,x),v(t,x,y))>(u_B(t,x),v_B(t,x,y))$ for all $t>0$ and $(x,y)\in\overline\Omega_B$.
  Thus, for all $0<A\le B$, it follows that 
 	\begin{equation*}
 	(U_B(x),V_B(x,y))\le \lim\limits_{t\to+\infty}\inf\limits_{0\le x\le ct, y\in[0,A]}(u(t,x),v(t,x,y))\le ({\nu}/{\mu},1).
 	\end{equation*}
   Passing to the limit $B\to+\infty$ together with Proposition \ref{prop3.9} (ii) implies that, for any $A>0$,
 	\begin{equation*}
 		\lim\limits_{ t\to+\infty}\inf\limits_{0\le x\le ct,0\le y\le A} (u(t,x),v(t,x,y))=({\nu}/{\mu},1).
 	\end{equation*}
 	The proof of Theorem \ref{thm-asp-half-plane} is thereby complete.
 	\end{proof}

Finally, we prove Theorem \ref{PTF-in half-plane} in the right direction, that is, problem \eqref{pb in half-plane} admits rightward pulsating fronts if and only if $c\ge c^*$. The proof is based on an asymptotic method.

\begin{proof}[Proof of Theorem \ref{PTF-in half-plane}]
	 Fix $c\ge  c^*$, one infers from \eqref{c*} that $c>c^*_R$ for any $R>R_0$.  It follows from Theorem \ref{thm-PTF-in strip} that the truncated problem \eqref{pb in strip} admits a rightward pulsating traveling front $(u_R(t,x),v_R(t,x,y))=(\phi_R(x-ct,x),\psi_R(x-ct,x,y))$ with wave speed $c$  in the strip $\overline\Omega_R$ connecting $(U_R,V_R)$ and $(0,0)$. Moreover, the profile $(\phi_R(s,x),\psi_R(s,x,y))$ is  decreasing in $s$ and $L$-periodic in $x$. Consider a sequence $(R_n)_{n\in\mathbb{N}}$ such that $R_n\to+\infty$ as $n\to+\infty$. Denote by $(\phi_{R_n}(s,x),\psi_{R_n}(s,x,y))$ the sequence of the rightward pulsating traveling fronts of \eqref{pb in strip} with speed $c$ and by $(U_{R_n},V_{R_n})$ the corresponding nontrivial steady states of \eqref{pb in strip} in the strips $\overline\Omega_{R_n}$. One has
\begin{align*}
\phi_{R_n}(-\infty,x)=U_{R_n}(x), ~~~~~~~ \phi_{R_n}(+\infty,x)=0,\cr
\psi_{R_n}(-\infty,x,y)=V_{R_n}(x,y), \ \psi_{R_n}(+\infty,x,y)=0,
\end{align*}
uniformly in $(x,y)\in\overline\Omega_{R_n}$. Moreover,  it follows from  Proposition \ref{prop3.9} that $0<U_{R_n}<{\nu}/{\mu}$ in $\mathbb{R}$, $0<V_{R_n}<1$ in $\mathbb{R}\times[0,R)$. By the limiting property in Proposition \ref{prop3.9} (ii), one can assume,  without loss of generality, that $\frac{4\nu}{5\mu}<U_{R_n}(\cdot) <\frac{\nu}{\mu}$ in $\mathbb{R}$ for each $n\in\mathbb{N}$.
Then due to the monotonicity and continuity of the function  $s\mapsto\phi_{R_n}(s,\cdot)$, there is a unique $s_n\in\mathbb{R}$ such that 
\begin{equation*}
\max\limits_{x\in\mathbb{R}}\phi_{R_n}(s_n,\cdot)=\max\limits_{x\in[0,L]}\phi_{R_n}(s_n,\cdot)=\frac{\nu}{2\mu}.
\end{equation*} 

Set $(\phi_n(s,x),\psi_n(s,x,y)):=(\phi_{R_n}(s+s_n,x),\psi_{R_n}(s+s_n,x,y))$.
Since
\begin{equation*}
\big(u_n(\frac{x-s}{c},x),v_n(\frac{x-s}{c},x,y)\big)=(\phi_n(s,x),\psi_n(s,x,y)),
\end{equation*}
by standard parabolic estimates, the sequence $((u_n, v_n))_{n\in\mathbb{N}}$ converges, up to extraction of a subsequence, locally uniformly to a classical solution $\big(u(\frac{x-s}{c},x), v(\frac{x-s}{c},x,y)\big)=(\phi(s,x),\psi(s,x,y))$ of \eqref{pb in half-plane} satisfying the normalization condition
\begin{equation*}
\max\limits_{x\in\mathbb{R}}\phi(0,\cdot)=\max\limits_{x\in[0,L]}\phi(0,\cdot)=\frac{\nu}{2\mu}.
\end{equation*}
Moreover, the profile  $(\phi(s,x),\psi(s,x,y))$ is non-increasing in $s$ and $L$-periodic in $x$ such that
\begin{align*}
\phi(-\infty,x)&={\nu}/{\mu}, ~~~~ \phi(+\infty,x)=0,\cr
\psi(-\infty,x,y)&=1, ~~~~~ \psi(+\infty,x,y)=0,
\end{align*}
uniformly in $x\in\mathbb{R}$ and locally uniformly in $y\in[0,+\infty)$.

Now, let us show the monotonicity of $(\phi(s,x),\psi(s,x,y))$ in $s$.  Since the pulsating front $(u(t,x),v(t,x,y))=(\phi(x-ct,x),\psi(x-ct,x,y))$ propagates with speed $c\ge c^*>0$, it follows that $u_t\ge 0$ for $t\in\mathbb{R}$ and $x\in\mathbb{R}$, $v_t\ge 0$ for $t\in\mathbb{R}$ and $(x,y)\in\overline\Omega$. Notice also that $(u(t,x), v(t,x,y))$ is a global classical solution of problem \eqref{pb in half-plane}, whence $z=v_t$ is a global classical  solution of $z_t=d\Delta z+f_v(x,v)z$  for $t\in\mathbb{R}$ and $(x,y)\in\Omega$ with $z\ge 0$. From the strong parabolic maximum principle, it follows that $z>0$ or $z\equiv 0$ for $t\in\mathbb{R}$ and $(x,y)\in\Omega$. That is, $v_t>0$ or $v_t\equiv 0$  for $t\in\mathbb{R}$ and $(x,y)\in\Omega$.
The latter case is impossible, otherwise one would derive from $v_t\equiv0$ that either $v\equiv 0$ or $v\equiv 1$  for $t\in\mathbb{R}$ and $(x,y)\in\Omega$. This is a contradiction with the limiting behavior of the pulsating fronts. Therefore, $v_t>0$  for $t\in\mathbb{R}$ and $(x,y)\in\Omega$ and by continuity $v_t>0$ for $t\in\mathbb{R}$ and $(x,y)\in\overline\Omega$. Likewise, one infers that $u_t>0$ for $t\in\mathbb{R}$ and  $x\in\mathbb{R}$. Hence, the rightward traveling fronts $(\phi(s,x),\psi(s,x,y))$ are decreasing in $s$.

 Assume that there exists a rightward pulsating traveling front $(\phi(x-ct,x),\psi(x-ct,x,y))$ of \eqref{pb in half-plane} with  speed $c>0$. Then, one infers from Theorem \ref{thm-asp-half-plane} that, for any $c'\in[0,c^*)$ and for any $B>0$,
\begin{equation*}
\lim\limits_{ t\to+\infty}\sup\limits_{0<x\le c't, y\in[0,B]}|(\phi(x-ct,x),\psi(x-ct,x,y))-({\nu}/{\mu},1)|=0.
\end{equation*}
In particular, for any $c'\in[0,c^*)$ and for any $B>0$, taking $x=c't$ and $y\in[0,B]$, there holds
\begin{equation*}
\lim\limits_{t\to+\infty}\phi((c'-c)t,c't)={\nu}/{\mu},\quad \lim\limits_{t\to+\infty}\psi((c'-c)t,c't,y)=1.
\end{equation*}
From the limiting condition \eqref{PTF-limit condition}, it follows that $c'<c$ for all $c'\in[0,c^*)$. Consequently, one gets $c^*\le c$. This implies  the non-existence of rightward pulsating traveling fronts with speed $0<c<c^*$.
\end{proof}

\section*{Acknowledgments}
I would like to express my sincerest gratitude to Professor François Hamel and Professor
Xing Liang for their continued guidance, support and encouragement during the preparation of
this project. I would like to thank Lei Zhang for many helpful discussions and thank Professor
Xuefeng Wang and the anonymous referee for their careful reading and for their valuable and
constructive comments which led to an important improvement of this manuscript.


\end{document}